\documentclass[11pt]{memoir}

\setlrmargins{*}{*}{1}
\checkandfixthelayout

\counterwithout{section}{chapter}
\counterwithout{figure}{chapter}
\firmlists

\usepackage{amssymb}

\makeatletter
\def\gacsparsep{6pt}

\newif\if@pfLongIndent 
\@pfLongIndentfalse
\newcommand{\pflongindent}{\@pfLongIndenttrue}
\newcommand{\pfshortindent}{\@pfLongIndentfalse}

\newlength{\pfindent}      
\newlength{\pftopsep}  
\newlength{\pfbotsep}  
\newlength{\pfsep}  
\newlength{\stepsep}        
\newbox{\pfbox}
\newcounter{pfhidelevel}
\newcommand{\pfshortnumbers}[1]{\pfshortNumberLevel=#1\relax}

\newcount\pfshortNumberLevel \pfshortNumberLevel=0

\newcommand{\pfstepnumber}[3]{%
   \ifnum \pfLevelCount < \pfshortNumberLevel
       #3%
     \else $\langle#1\rangle#2$%
   \fi}

\newcommand{\pflevelnumber}[2]{%
   \ifnum \pfLevelCount < \pfshortNumberLevel
       #2%
     \else $\langle#1\rangle$%
   \fi}

\newif\if@pfSideNumbers 
\@pfSideNumbersfalse
\newcounter{pf@sidenumberdepth}
\newlength{\pf@sidenumberoutdent}
\newcommand{\pfsidenumbers}[2]{\@pfSideNumberstrue
     \setcounter{pf@sidenumberdepth}{#1}%
     \addtocounter{pf@sidenumberdepth}{-1}%
     \setlength{\pf@sidenumberoutdent}{#2}}
\newcommand{\pfnosidenumbers}{\@pfSideNumbersfalse}

\newcount\pfLevelCount  \pfLevelCount=0 
\newcount\pfStepCount   \pfStepCount=0  
\newcommand{\pfLongLevel}{} 
\newcommand{\pfLongStep}{}  
\newcommand{\pfStepName}{}  

\newcommand{\pfSetName}{%
  \edef\pfLongStep{%
    \ifnum\pfLevelCount>\@ne
      \pfLongLevel\pfdot\the\pfStepCount
     \else\the\pfStepCount\fi}%
  \edef\pfStepName{%
   \ifnum \pfLevelCount < \pfshortNumberLevel
    {\pfLongStep}{\pfLongStep}%
    \else
     {$\langle\the\pfLevelCount\rangle$}%
     {$\langle\the\pfLevelCount\rangle\the\pfStepCount$}%
   \fi}}

\newcommand{\pfSetRef}[1]{%
  \@ifundefined{pf@#1}%
    {\expandafter
      \edef\csname pf@#1\endcsname{\pfStepName}}%
    {\typeout{WARNING:
         proof step "#1" (<\the\pfLevelCount>\the\pfStepCount) 
        already defined}}%
    }

\newcommand{\pfPrintStepNumber}[2]{#2}
\newcommand{\pfPrintLevelNumber}[2]{#1}

\newcommand{\stepref}[1]{\@ifundefined{pf@#1}{{\bf ??}\typeout{WARNING:
 proof step "#1" undefined}}{\expandafter\expandafter\expandafter
 \pfPrintStepNumber\csname pf@#1\endcsname}}

\newcommand{\levelref}[1]{\@ifundefined{pf@#1}{{\bf ??}\typeout{WARNING:
 proof step #1 undefined}}%
  {\expandafter\expandafter\expandafter
 \pfPrintLevelNumber\csname pf@#1\endcsname}}

\newlength{\pf@outdent}
\newcommand{\pfSideNumber}{%
  \if@pfSideNumbers 
   \ifnum\pfLevelCount>\value{pf@sidenumberdepth}%
    \hspace*{-\pf@outdent}%
    \makebox[0pt][l]{\footnotesize\pfLongStep}%
    \hspace*{\pf@outdent}%
    \else\fi
   \else\fi}

 \newcounter{branch}%

 \newcounter{branchi}[branch]

 \newcounter{branchii}[branchi]

 \newcounter{branchiii}[branchii]

 \newcounter{branchiv}[branchiii]

 \newcounter{branchv}[branchiv]

 \newcounter{branchvi}[branchv]

 \def\CurBranch{{branch\romannumeral\pfLevelCount}}%

 \def\TheCurBranch{\csname thebranch\romannumeral\pfLevelCount\endcsname}%

 \def\RefStepCounter#1{\refstepcounter{#1}} 

 \def\SetTheBranch#1#2{%
  \expandafter\def\csname thebranch#2\endcsname{%
    \csname thebranch#1\endcsname\pfdot\arabic{branch#2}%
   }}

 \def\OpenLevel{%
  \edef\ParentBranch{\CurBranch}
  \edef\romanPrevLevel{{\romannumeral \pfLevelCount}}
  \advance\pfLevelCount\@ne%
  \edef\romanLevel{{\romannumeral\pfLevelCount}}\@ifundefined
{c@branch\romannumeral\pfLevelCount}{\typeout
{WARNING: too many prooof levels}}{}\expandafter\setcounter\CurBranch{0}
}

\def\pflist#1#2{\ifnum \@listdepth >15\relax \@toodeep 
     \else \global\advance\@listdepth\@ne \fi
  \rightmargin \z@ \listparindent\z@ \itemindent\z@
  \csname @list\romannumeral\the\@listdepth\endcsname 
  \def\@itemlabel{#1}\let\makelabel\@mklab \@nmbrlistfalse #2\relax
  \@trivlist
  \parskip\parsep \parindent\listparindent
  \advance\linewidth -\rightmargin \advance\linewidth -\leftmargin
  \advance\@totalleftmargin \leftmargin
  \parshape \@ne \@totalleftmargin \linewidth 
  \ignorespaces}

\def\endpflist{\global\advance\@listdepth\m@ne
    \endtrivlist}

\let\@listvii=\@listv
\let\@listviii=\@listv
\let\@listix=\@listv
\let\@listx=\@listv
\let\@listxi=\@listv
\let\@listxii=\@listv
\let\@listxiii=\@listv
\let\@listxiv=\@listv
\let\@listxv=\@listv
\let\@listxvi=\@listv

\newenvironment{prooof}{
  \edef\pfLongLevel          
   {\ifnum\pfLevelCount>\z@  
     \ifnum\pfLevelCount>\@ne
      \pfLongLevel\pfdot\else\fi  
      \the\pfStepCount       
     \else\fi}
  \OpenLevel                 
  \@tempcnta=\value{pfhidelevel}%
  \ifnum\@tempcnta<\@ne
    \setcounter{pfhidelevel}{1}%
    \typeout{WARNING: pfhidelevel < 1, setting to 1}%
    \@tempcnta=\@ne\fi
  \advance\@tempcnta\@ne
  \if@qedstep
    \advance\@tempcnta\@ne
    \ifnum\pfLevelCount 
        = \@tempcnta 
     \sbox{\pfbox}\bgroup\begin{minipage}{\textwidth}\fi
  \else     
   \ifnum\pfLevelCount 
        = \@tempcnta 
     \sbox{\pfbox}\bgroup\begin{minipage}{\textwidth}\fi\fi
  \pfStepCount=\z@           
  \ifnum\pfLevelCount>\@ne   
   \begin{pflist}{}{
    \topsep=\pfsep\relax     
    \itemsep=\z@             
     \parsep=\gacsparsep             
    \partopsep=\z@           
    \if@pfLongIndent         
     \settowidth{\leftmargin}
      {\expandafter          
       \pfPrintStepNumber    
       \pfStepName.}%
     \advance\leftmargin
     \labelsep 
    \else                    
    \leftmargin=\pfindent    
    \fi\relax        
   }   \item[]
   \else \par                
    \addvspace{\pftopsep}
    \parindent=\z@           
    \parskip = \z@           
    \@ifundefined{mathindent}{%
    }{%
      \mathindent=1em}%
 \fi
\@qedstepfalse
}
  {\ifnum\pfLevelCount>\@ne  
    \end{pflist}
   \else \par                
     \addvspace{\pfbotsep}
   \fi
  \@tempcnta=\value{pfhidelevel}%
  \advance\@tempcnta\@ne
  \if@qedstep
    \advance\@tempcnta\@ne
    \ifnum\pfLevelCount 
        = \@tempcnta 
     \end{minipage}\egroup\sbox{\pfbox}{}\@qedstepfalse\fi
  \else 
      \ifnum\pfLevelCount 
        =\@tempcnta 
        \end{minipage} 
        \egroup\sbox{\pfbox}{}\fi
 \fi}

  {\par}                     

\newcommand{\step}[2]{\begin{step+}{#1}#2\end{step+}}
\newif\if@qedstep 
\@qedstepfalse

\newif\if@unhideqedstep 

\@unhideqedstepfalse
\newcommand{\unhideqedprooof}{\@unhideqedsteptrue}
\newcommand{\hideqedprooof}{\@unhideqedstepfalse}

\newcommand{\qedstep}{\step{qedstep\the\pfLevelCount}{Q.E.D.}
\if@unhideqedstep\@qedsteptrue\fi}

\newcommand{\nostep}[1]{%
  \expandafter\RefStepCounter\CurBranch    
  \label{#1}                               
  \advance\pfStepCount\@ne       
  \pfSetName                     
  \pfSetRef{#1}
  }

\newenvironment{step+}[1]%
 {\endgroup                      
  \expandafter\RefStepCounter\CurBranch    
  \label{#1}                               
  \advance\pfStepCount\@ne       
  \pfSetName                     
  \pfSetRef{#1}
 \begingroup\@endpefalse         
   \def\@currenvir{step+}
 \begin{pflist}{}{
   \setlength
     {\pf@outdent}{\textwidth}
   \addtolength                  
    {\pf@outdent}{-\linewidth}%
   \addtolength
    {\pf@outdent}%
    {\pf@sidenumberoutdent}%
   \topsep=\stepsep\relax        
   \itemsep=\z@                  
   \parsep=\gacsparsep                   
   \partopsep=\z@                
   \settowidth{\labelwidth}
     {\expandafter
      \pfPrintStepNumber
      \pfStepName.}%
   \leftmargin=\labelwidth\relax 
   \advance\leftmargin\labelsep 
   \relax}%
 \item[\pfSideNumber
  \expandafter                       
  \pfPrintStepNumber\pfStepName.]}%
 {\end{pflist}}

\newenvironment{assume+}{%
 \begin{pflist}{}{
   \topsep=\z@                   
   \itemsep=\z@                  
   \parsep=\gacsparsep                   
   \partopsep=\z@                
   \settowidth{\labelwidth}
     {{\kwfont Assume\/}:}%
   \leftmargin=\labelwidth\relax 
   \advance\leftmargin\labelsep  
   \relax}%
 \item[{\kwfont Assume\/}:]}
 {\end{pflist}}

\newenvironment{prove+}{%
 \begin{pflist}{}{
   \topsep=\z@                   
   \itemsep=\z@                  
   \parsep=\gacsparsep   
   \partopsep=\z@                
   \settowidth{\labelwidth}
     {{\kwfont Assume\/}:}%
   \leftmargin=\labelwidth\relax 
   \advance\leftmargin\labelsep  
   \relax}%
 \item[{\kwfont 
          Prove\/}:\hfill]}
 {\end{pflist}}

\newenvironment{pflet+}{%
 \begin{pflist}{}{
   \topsep=\z@                   
   \itemsep=\z@                  
   \parsep=\gacsparsep 
   \partopsep=\z@                
   \settowidth{\labelwidth}
     {{\kwfont Let\/}:}%
   \leftmargin=\labelwidth\relax 
   \advance\leftmargin\labelsep  
   \relax}%
 \item[{\kwfont 
          Let\/}:\hfill]}
 {\end{pflist}}

\newcommand{\case}[1]{%
 \begin{pflist}{}{
   \topsep=\z@                   
   \itemsep=\z@                  
   \parsep=\gacsparsep 
   \partopsep=\z@                
   \settowidth{\labelwidth}
     {{\kwfont Case\/}:}%
   \leftmargin=\labelwidth\relax 
   \advance\leftmargin\labelsep  
   \relax}%
 \item[{\kwfont
           Case\/}:]             
 #1
 \end{pflist}}

\newcommand{\pf}{{\kwfont Proof\/}.} 

\let\kwfont=\scshape
\def\pfdot{.}

\def\@push#1#2{{\let\@nil\relax\let\@elt\relax\xdef#1{#2\@elt#1}}}
\def\@pop#1#2{{\let\@nil\relax\let\@elt\relax
              \xdef#2{\expandafter\@innerhead#1}
              \xdef#1{\expandafter\@innerpop#1}}}
\def\@innerpop#1\@elt#2\@nil{#2\@nil}
\def\@head#1#2{{\let\@elt\relax\xdef#2{\expandafter\@innerhead#1}}}
\def\@innerhead#1\@elt#2\@nil{#1}
\def\newstack#1{\gdef#1{\@nil}}

\newcounter{pf@conjCounter}   
\newstack\pf@conj             

\newenvironment{conj*}{%
 \@push\pf@conj{\the\value{pf@conjCounter}}%
 \setcounter{pf@conjCounter}{0}%
 \begin{array}[t]{@{\addtocounter{pf@conjCounter}{1}%
   \mbox{\protect\small\protect\arabic{pf@conjCounter}.}
   \land\;}l@{}}%
 }{%
 \end{array}%
 \@pop\pf@conj\pf@temp
  \setcounter{pf@conjCounter}{\pf@temp}}

\newenvironment{disj*}{%
 \@push\pf@conj{\the\value{pf@conjCounter}}%
 \setcounter{pf@conjCounter}{0}%
 \begin{array}[t]{@{\addtocounter{pf@conjCounter}{1}%
   \mbox{\protect\small\protect\alph{pf@conjCounter}.}
   \lor\;}l@{}}%
 }{%
 \end{array}%
 \@pop\pf@conj\pf@temp
  \setcounter{pf@conjCounter}{\pf@temp}}

\newcounter{pfenum}
\newcounter{pfenumdepth}
\newlength{\enumindent}
\@definecounter{pfenumi}
\@definecounter{pfenumii}
\@definecounter{pfenumiii}

\def\thepfenumi{\arabic{pfenumi}}

\def\thepfenumii{\alph{pfenumii}}
\def\p@pfenumii{\thepfenumi}

\def\p@pfenumiii{\thepfenumi\thepfenumii}

\newcommand{\pf@setEnumWidth}[1]{%
  \settowidth{#1}{\setcounter{\@pfenumctr}{2}%
  \csname the\@pfenumctr\endcsname.%
  \setcounter{\@pfenumctr}{0}}}

\newcommand{\pf@enumLabel}{%
  \hfill\makebox[0pt][r]{\csname the\@pfenumctr\endcsname.}}

\newenvironment{pfenum*}{
  \ifnum \value{pfenumdepth}>2
    \relax\@toodeep \else        
  \addtocounter{pfenumdepth}{1}
  \edef\@pfenumctr{pfenum
    \romannumeral\the            
    \value{pfenumdepth}}
   \fi                           
  \begin{pflist}
  {\pf@enumLabel}{
   \labelsep=                    
    \ifcase\value{pfenumdepth}   
      \labelsep                  %
     \or .67\labelsep            
     \or .67\labelsep            
     \else \labelsep             
     \fi                         %
   \topsep=\z@                   
   \itemsep=\z@                  
   \parsep=\gacsparsep 
   \partopsep=\z@                
   \pf@setEnumWidth\labelwidth   
   \leftmargin=\labelwidth\relax 
   \advance\leftmargin\labelsep  
   \advance\leftmargin\enumindent
   \relax
   \usecounter{\@pfenumctr}
 }%
 }{
   \end{pflist}
   \addtocounter{pfenumdepth}{-1}
 }

\def\makefcn#1#2#3{{\let\or\relax
     \gdef\fcn@temp{}%
     \gdef\fcn@tempb{#2}%
     \@tfor\foo:=#3\do{\xdef\fcn@temp{\fcn@temp\or\foo}\xdef\fcn@tempb{\foo}}%
     \let\ifcase\relax\let\else\relax\let\fi\relax
     \xdef#1##1{\ifcase ##1 #2\fcn@temp\else\fcn@tempb\fi}}}

\newcommand{\pfmathdef}[1]{\relax\ifmmode #1\else $#1$\fi}

\newenvironment{pproof}{\begin{prooof}\pf\ }{\end{prooof}}

\newenvironment{Proof}[1][\proofname]{\vspace{0pt}\begin{proof}[#1]\refstepcounter{branch}\begin{prooof}}%
  {~\qed\end{prooof}\renewcommand{\qed}{}\end{proof}} 

\makeatother 

\usepackage[hyperindex,colorlinks,citecolor=blue]{hyperref}

\usepackage[all]{hypcap} 

\usepackage{float}

\makeatletter

\usepackage{mathtools}
\usepackage{amsthm}

\newcommand{\myTimes}{
\usepackage{txfonts}
  \DeclareSymbolFont{cmsymbols}     {OMS}{cmsy}{m}{n}
  \DeclareSymbolFontAlphabet{\mathcal}   {cmsymbols}
\DeclareMathAlphabet{\mathbb}{U}{msb}{m}{n}
}

\def\@@enum@[#1]{%
  \@enLab{}\let\@enThe\@enQmark
  \@enloop#1\@enum@
  \expandafter\edef\csname label\@enumctr\endcsname{\the\@enLab}%
  \expandafter\let\csname the\@enumctr\endcsname\@enThe
  \csname c@\@enumctr\endcsname7
  \expandafter\settowidth
            \csname leftmargin\romannumeral\@enumdepth\endcsname
            {\the\@enLab\hspace{\labelsep}}%
  \@enum@}

\newenvironment{alphenum} 
 {\begin{enumerate}[{\upshape (a)}]}{\end{enumerate}}

 {\begin{enumerate}[{\upshape (a)}]}{\end{enumerate}}

 \newcommand{\df}[1]{\textit{\textbf{#1}}}

\newcommand{\@hideone}[1]{\ifnum #1=1\else\@arabic#1\fi}
 \newcommand{\Pageno}{\c@page}

\newcommand{\customqed}[1]{{\renewcommand{\qedsymbol}{#1}\qed}}
\newcommand{\varqed}{\customqed{\hbox{$\lrcorner$}}}

\newcommand{\intheoremstyle}{
 \newtheorem{lemma}{Lemma}[section]

\newtheorem{theorem}{Theorem}

 \newtheorem{corollary}[lemma]{Corollary}

}

\newcommand{\indefinitionstyle}{
\newtheorem{Definition}[lemma]{Definition}
\newenvironment{definition}{\begin{Definition}}{\varqed\end{Definition}}

 \newtheorem{Notation}[lemma]{Notation}

 \newtheorem{Condition}[lemma]{Condition}
 \newenvironment{condition}{%
   \begin{Condition}}{\varqed\end{Condition}}

 \newtheorem{Remark}[lemma]{Remark}
 \newtheorem{Remarks}[lemma]{Remarks}
 
 \newenvironment{remark}{%
   \begin{Remark}}{\varqed\end{Remark}}
 \newenvironment{remarks}{%
   \begin{Remarks}}{\varqed\end{Remarks}}

 \newtheorem{Example}[lemma]{Example}
 \newtheorem{Examples}[lemma]{Examples}

 \newenvironment{example}{%
   \begin{Example}}{\varqed\end{Example}}

 \theoremstyle{remark}
}

\newcommand{\standardTheorems}{
\theoremstyle{plain}
\intheoremstyle
\theoremstyle{definition}
\indefinitionstyle
}

 \newcommand \ag{\alpha}

 \newcommand \eps{\varepsilon}

 \renewcommand \lg{\lambda} 

 \newcommand \sg{\sigma}

 \newcommand{\cB}{\mathcal{B}}
 \newcommand{\cC}{\mathcal{C}}
 
 \newcommand{\cE}{\mathcal{E}}
 \newcommand{\cF}{\mathcal{F}}
 \newcommand{\cG}{\mathcal{G}}

 \newcommand{\cL}{\mathcal{L}}
 \newcommand{\cM}{\mathcal{M}}

 \newcommand{\cQ}{\mathcal{Q}}
 
 \newcommand{\cS}{\mathcal{S}}
 \newcommand{\cT}{\mathcal{T}}

 \newcommand{\cW}{\mathcal{W}}

 \newcommand{\bbZ}{\mathbb{Z}}

 \newcommand\set[1]{\mathopen\{#1\mathclose\}}
 \newcommand\setof[1]{\mathopen\{\,#1\,\mathclose\}}

 \newcommand\tup[1]{( #1)}
\newcommand{\pair}[2]{(#1,#2)}

 \newcommand\ang[1]{{\mathopen\langle #1\mathclose\rangle}}

 \newcommand\bigparen[1]{{\bigl(\,#1\,\bigr)}}

 \newcommand\paren[1]{( #1)}

 \newcommand{\cei}[1]{{\lceil #1\rceil}}
 \newcommand{\flo}[1]{{\lfloor #1\rfloor}}
 \newcommand{\Cei}[1]{{\left\lceil #1\right\rceil}}
 \newcommand{\Flo}[1]{{\left\lfloor #1\right\rfloor}}

 \DeclareMathOperator{\Prob}{\mathbb{P}}
 \newcommand\Pbof[1]{\Prob\mathopen\{\,#1\,\mathclose\}}

 \newcommand\Maj{\mathop{\operator@font Maj}\nolimits}
 \newcommand\nor{\mathbin{\operator@font NOR}}

 \let\imp=\Rightarrow

 \newcommand {\sbsq}{\subseteq}

 \newcommand {\spsq}{\supseteq}

\let\kwfont\itshape

\@ifundefined{stepsep}{}{\setlength{\stepsep}{1ex}}  
 
\@ifundefined{pfshortnumbers}{}{\pfshortnumbers{4}}

 \newcommand{\txt}[1]{\text{\rmfamily\mdseries\upshape{#1}}}

\newcommand{\killtext}[1]{}

 \newcommand{\half}{{\textstyle \frac{1}{2}}}

 \newcommand{\fourth}{{\textstyle \frac{1}{4}}}

\newcommand{\verythinmathskip}{\mskip 1mu}

\newcommand{\clint}[2]{[\verythinmathskip #1,#2\verythinmathskip]}
\newcommand{\lint}[2]{[\verythinmathskip #1,#2)}

\newcommand{\rint}[2]{(#1,#2\verythinmathskip]}

\newcommand\amod{\mskip-\medmuskip\mkern5mu\mathbin
  {\operator@font amod}\penalty900
  \mkern5mu\mskip-\medmuskip}

\standardTheorems

\makeatother 

\mathtoolsset{mathic}

\myTimes

\theoremstyle{definition}

 \newtheorem{Defstep}{Step}

 \newenvironment{defstep}{%
   \begin{Defstep}}{\varqed\end{Defstep}}

\numberwithin{equation}{section} 

 \newcommand{\aux}{c}
\newcommand{\bub}{\varDelta}
\newcommand{\bubxp}{\delta}
 \renewcommand{\d}{d}
 \newcommand{\D}{D}
 \newcommand{\f}{\varPhi}
 \newcommand{\fxp}{\varphi}
 \newcommand{\g}{\varGamma}
 \newcommand{\gf}{\varPsi}

 \newcommand{\gxp}{\gamma}
 \newcommand{\h}{h}
\newcommand{\hxp}{\chi}

 \renewcommand{\L}{L}
 \newcommand{\m}{\tilde m} 
 \newcommand{\p}{p}
 \renewcommand{\r}{r}
 \newcommand{\R}{R}
 \newcommand{\Rect}{\text{\rmfamily\mdseries\upshape Rect}}
 \newcommand{\s}{s}
 \newcommand{\slb}{\sg}
 \newcommand{\slblb}{\slb}
\newcommand{\slope}{\text{\rmfamily\mdseries\upshape slope}}
 \newcommand{\slopeincr}{\varLambda}
 \renewcommand{\t}{\square}
 \newcommand{\T}{T} 
 \newcommand{\tub}{w}
\newcommand{\tubxp}{\omega}
\newcommand{\txp}{\tau}
\newcommand{\txpub}{\overline{\tau}}
\newcommand{\ncln}{q}

\newcommand{\w}{\triangle}

 \newcommand{\Body}{\text{\rmfamily\mdseries\upshape Body}}
 \newcommand{\Wvalues}{\text{\rmfamily\mdseries\upshape Wvalues}}

\begin{document}

\title{Clairvoyant embedding in one dimension}

\author{Peter G\'acs
\\ Boston University
\\ gacs@bu.edu}

\maketitle
\thispagestyle{empty}

\begin{abstract} 
 Let \( v, w\) be infinite 0-1 sequences, and \( \m \) a positive integer.  
We say that \( w \)
is \( \m \)-\df{embeddable} in \( v \), if there exists an increasing sequence
  \( (n_{i} : i \ge 0) \)  of integers with \( n_{0}=0 \), such that 
\( 1\le n_{i} - n_{i-1} \le \m \), \( w(i) = v(n_i) \) for all \( i \ge 1 \).  
Let \(  X \) and \( Y \) be coin-tossing sequences.  
We will show that
there is an \( \m \) with the property that \( Y \) is \( \m \)-embeddable
into \( X \) with positive probability.  
This answers a question that was open for a while.
The proof generalizes somewhat the hierarchical method of an
earlier paper of the author on dependent percolation.  
\end{abstract}

\section{Introduction}
\begin{sloppypar}
Consider the following problem, stated 
in~\cite{GrimmettLiggettRichthammer08,GrimmettThreeProblems09}.
Let \( v=(v(1),v(2)\dots),\) \( w=(w(1),w(2)\dots) \) be infinite 0-1 sequences, 
and \( \m>0\).
We say that \( w \) is \( \m \)-\df{embeddable} in \( v \), 
if there exists an increasing sequence \( (n_{i} : i \ge 1) \) of positive integers
such that \( w(i) = v(n_i) \), and \( 1\le n_{i} - n_{i-1} \le \m \) for all \( i \ge 1 \).
(We set \( n_{0} = 0 \), so \( n_{1}\le \m \) is required.) 
Let \( X=(X(1),X(2),\dots) \) and \( Y=(Y(1),Y(2),\dots) \) 
be sequences of independent Bernoulli variables with parameter 1/2.
The question asked was whether there is any \( \m \) with the property that
if \( Y \) is independent of \( X \) then it is
\( \m \)-embeddable into \( X \) with positive probability.
The present paper answers the question positively.
\end{sloppypar}

\begin{theorem}\label{thm:main}
There is an \( \m \) with the property that if \( Y \) is independent of \( X \) then it is
\( \m \)-embeddable into \( X \) with positive probability.  
\end{theorem}

It turns out that independence is not needed, see Theorem~\ref{thm:non-indep} below.

The proof allows the computation of an upper bound on \( \m \), but we will not
do this, and not just to avoid ridicule: many steps of the proof
would become less transparent when trying to do this.

Here is a useful equivalent formulation.
First we define the fixed directed graph 
\( G_{\m}=(\bbZ_{+}^{2},E) \).
From each point \( \pair{i}{j} \) edges go to
\( \pair{i+1}{j+1} \), \( \pair{i+2}{j+1} \), \( \dots \), \( \pair{i+\m}{j+1} \).
The random graph 
\begin{align}\label{eq:Gbub}
   \cG_{\m}(X,Y)=(\bbZ_{+}^{2},\cE)
 \end{align}
is defined as follows: delete all edges going \emph{into}
points \( \pair{i}{j} \) of \( G_{\m} \) with \( X(i)\ne Y(j) \).
(In percolation terms, in \( \cG_{\m} \) we would call a point ``open''
if it has some incoming edge.)
Now \( Y \) is embeddable into \( X \) if and only if there is an infinite path in \( \cG_{\m} \)
starting at the origin.
So the embedding question is equivalent to a percolation question.

Our proof generalizes slightly the method of~\cite{GacsWalks11}, making also its
technical result more explicit.
Just before uploading to the arXiv, the author learned that Basu and Sly have
also proved the embedding theorem, in an independent work~\cite{BasuSly12}.
They are citing another, simultaneous and independent, paper by Sidoravicius.

The proof of Theorem~\ref{thm:main} relies on the independence of the processes
\( X \) and \( Y \).
But the proof in~\cite{BasuSly12}, just like the proof of the compatible sequences 
result in~\cite{GacsChat04}, does not: it applies to any joint distribution with 
the coin-tossing marginals \( X,Y \).
Allan Sly showed in~\cite{SlyEmail13feb} how Theorem~\ref{thm:main} 
can also be adapted to prove a version without the independence assumption:

\begin{theorem}\label{thm:non-indep}
There is an \( \m \) with the property that if \( Y \) has a joint distribution with 
\( X \) then it is \( \m \)-embeddable into \( X \) with positive probability.
\end{theorem}
\begin{proof}
It is easy to derive from Theorem~\ref{thm:main}, but is even more immediate from
the proof as pointed out in Remark~\ref{rem:pos-prob-small} below,
that the probability of the existence of an
\( \m \)-embedding converges to 1 as \( \m\to\infty \).
Let us choose an \( \m \) now making this probability at least \( 1-\eps \) for some
\( \eps<1/2 \).

Given two coin-tossing sequences \( X,Y \) with a joint distribution,
let us create a coin-tossing sequence \( Z \) independent of \( \pair{X}{Y} \).
The above remark implies that there is an \( \m \) such that
\( Y \) is \( \m \)-embeddable into \( Z \) with 
probability \( >1-\eps \), and \( Z \) is \( \m \)-embeddable into \( X \) with 
probability \( >1-\eps \).
Combining the two embeddings gives an \( \m^{2} \)-embedding of \( Y \) into \( X \)
with probability \( > 1-2\eps \).
\end{proof}

It is unknown currently whether the theorem of~\cite{GacsWalks11} on clairvoyant
scheduling of random walks can also be generalized to 
non-independent random walks.

Just like in~\cite{GacsWalks11}, we will introduce several
extra elements into the percolation picture.
For consistency with what comes later, let us call open points 
``{lower left trap-clean}''.
Let us call any interval \( \rint{i}{i+\m} \) a ``{vertical wall}'' if
\( X(i+1)=X(i+2) \) \( =\dots= \) \( X(i+\m) \).
We call an interval \( \rint{a}{a+1} \) ``horizontal hole fitting'' this wall if
\( Y(a+1)=X(i+1) \).
The idea is that a vertical wall forms a certain obstacle for a path
\( n_{1},n_{2},\dots \). 
But if the path arrives at a fitting hole \( \rint{a}{a+1} \), that is it has \( n_{a}=i \),
then it can pass through, with \( n_{a+1}=i+\m \).
The vertical walls are obstacles to paths, but there is hope: a wall has only
probability \( 2^{-\m+1} \) to start at any one place, while a hole fitting it has
probability 1/2 to start at any place.
Under appropriate conditions then, walls can be passed.
The failure of these conditions gives rise to a similar, ``higher-order'' model
with a new notion of walls.
It turns out that in higher-order models, some more types of 
element (like traps) are needed.
This system was built up in
the paper~\cite{GacsWalks11}, introducing a model called ``mazery''.
We will generalize mazeries slightly (more general bounds on slopes and
cleanness), to make them applicable to the embedding situation.

It is an understatement to say that the construction and proof
in~\cite{GacsWalks11} are complex.
Fortunately, much of it carries over virtually without changes, only 
some of the proofs (a minority) needed to be rewritten.
On the other hand, giving up any attempt to find reasonable bounds on 
\( \m \) made it possible to simplify some parts; in particular, the proof of the
Approximation Lemma (Lemma~\ref{lem:approx}) is less tedious here
than in~\cite{GacsWalks11}.

We rely substantially on~\cite{GacsWalks11}
for motivation of the proof structure and illustrations.
Each lemma will still be stated, but we will omit the proof of those that did not
change in any essential respect.

The rest of the paper is as follows.
Section~\ref{sec:mazery} defines mazeries.
Section~\ref{sec:main-lemma} formulates the main theorem and main lemma in terms
of mazeries, from which Theorem~\ref{thm:main} follows.
Section~\ref{sec:plan} defines the scale-up operation \( \cM^{k}\mapsto\cM^{k+1} \).
It also proves that scale-up preserves almost all \emph{combinatorial}
properties, that is those that do not involve probability bounds.
The one exception is the reachability property, formulated by
Lemma~\ref{lem:approx} (Approximation): its proof is postponed to
Section~\ref{sec:approx-proof}.
Section~\ref{sec:params} specifies the parameters in a way that guarantees
that the probability conditions are also preserved by scale-up.
Section~\ref{sec:bounds} estimates how the probability bounds are transformed by
the scale-up operation.
Section~\ref{sec:main-proof} proves the main lemma.

\section{Mazeries}\label{sec:mazery}

This section is long, and is very similar to Section 3
in~\cite{GacsWalks11}: we will point out the differences.

 \subsection{Notation}\label{subsec:notation}

The notation \( \pair{a}{b} \) for real numbers \( a,b \) will generally mean for us
the pair, and not the open interval.
Occasional exceptions would be pointed out, in case of ambiguity.
We will use
 \[
a \land b = \min(a,b),\quad a \lor b = \max(a,b).
 \]
To avoid too many parentheses, we use the convention 
 \begin{align*}
   a\land b\cdot c = a\land(b\cdot c).
 \end{align*}
We will use intervals on the real line 
and rectangles over the Euclidean plane, even though we are really only
interested in the lattice \( \bbZ_{+}^{2} \).
To capture all of \( \bbZ_{+} \) this way, 
for our right-closed intervals \( \rint{a}{b} \), we allow the
left end \( a \) to range over all the values \( -1,0,1,2,\dots \).
For an interval \( I=\rint{a}{b} \), we will denote 
 \[
   X(I) = \tup{X(a+1),\dots,X(b)}.
 \]
The \df{size} of an interval \( I \) with endpoints \( a,b \) (whether it is open,
closed or half-closed), is denoted by \( |I| = b-a \).
By the \df{distance} of two points 
\( a = \pair{a_{0}}{a_{1}} \), \( b=\pair{b_{0}}{b_{1}} \) of the plane, we mean 
 \[
   |b_{0}-a_{0}| \lor |b_{1}-a_{1}|.
 \]
The \df{size} of a rectangle
 \[
   \Rect(a, b) 
    = \clint{a_{0}}{b_{0}} \times \clint{a_{1}}{b_{1}}
 \]
in the plane is defined to be equal to the distance between \( a \) and \( b \).
For two different points \( u = \pair{u_{0}}{u_{1}} \), \( v = \pair{v_{0}}{v_{1}} \)
in the plane, when \( u_{0} \le v_{0} \), \( u_{1} \le v_{1} \):
 \begin{alignat*}{2}
      &\slope(u, v) &&= \frac{v_{1}-u_{1}}{v_{0}-u_{0}}.
  \end{alignat*}
We introduce the following partially open rectangles
 \begin{equation*}
 \begin{alignedat}{2}
   &\Rect^{\rightarrow}(a, b)
                      &&= \rint{a_{0}}{b_{0}} \times \clint{a_{1}}{b_{1}},
\\ &\Rect^{\uparrow}(a, b)
                      &&= \clint{a_{0}}{b_{0}} \times \rint{a_{1}}{b_{1}}.
 \end{alignedat}
 \end{equation*}
The relation
 \[
 u \leadsto v
 \]
says that point \( v \) is reachable from point \( u \) (the underlying graph
will always be clear from the context).
For two sets \( A, B \) in the plane or on the line,
 \[
   A + B = \setof{a + b : a \in A,\; b \in B}.
 \]

\subsection{The structure}

A mazery is a special type of random structure we are about to define.
Eventually, an infinite series of mazeries \( \cM^{1} \), \( \cM^{2}\),
\( \dots \) will be defined.
Each mazery \( \cM^{i} \) for \( i>1 \) will be obtained from the preceding one by a
certain scaling-up operation.
Mazery \( \cM^{1} \) will derive directly from the original percolation
problem, in Example~\ref{xmp:base}.

\subsubsection{The tuple}

All our structures defined below refer to ``percolations'' over the same lattice
graph \( \cG=\cG(X,Y) \) depending on the coin-tossing sequences \( X,Y \).
It is like the graph \( \cG_{3\m}(X,Y) \) introduced in~\eqref{eq:Gbub} above,
but we will not refer to \( \m \) explicitly.

A \df{mazery}
 \begin{equation*}
  (\cM, \slblb, \slb_{x},\slb_{y},\R,\bub,\tub, \ncln_{\w},\ncln_{\t}) 
 \end{equation*}
consists of a random process \( \cM \), and the listed nonnegative parameters.
Of these, \( \slblb,\slb_{x},\slb_{y} \) are called \df{slope lower bounds},
\( \R \) is called the \df{rank lower bound}, and they satisfy
 \begin{align}
\label{eq:slb-lb} 1/2\R &\le \slb_{x}/2 
\le \slblb\le \slb_{x}, 
\quad 2\le\slb_{y},
\\ \label{eq:slb-ub}   \slb_{x}\slb_{y}&< 1-\slblb.
 \end{align}
With~\eqref{eq:slb-lb} this implies \( \slb \le \slb_{x} < \frac{1-\slblb}{2} \).
We call \( \bub \) the \df{scale parameter}.
We also have the probability upper bounds \( \tub, \ncln_{j}>0 \) with 
 \begin{align*}
   \ncln_{\w} < 0.05,\quad
\ncln_{\t} < 0.55,
 \end{align*}
which will be detailed below, along with conditions that they must satisfy.
(In~\cite{GacsWalks11}, there was just one parameter \( \slb \) and 
one parameter \( \ncln \).)
Let us describe the random process
 \[
   \cM  = (X, Y, \cG, \cT, \cW, \cB, \cC, \cS).
 \]
In what follows, when we refer to the mazery, we will just identify it with \( \cM \).
We have the random objects
 \begin{equation*}
   \cG,\quad
 \cT, \quad
 \cW = \tup{\cW_{x},\cW_{y}},\quad
 \cB = \tup{\cB_{x},\cB_{y}},\quad
 \cC = \tup{\cC_{x}, \cC_{y}},\quad
 \cS = \tup{\cS_{x}, \cS_{y}, \cS_{2}}.
 \end{equation*}
all of which are functions of \( X,Y \).
The graph \( \cG \) is a random graph.

 \begin{definition}[Traps]
In the tuple \( \cM \) above, \( \cT \) is a random set 
of closed rectangles of size \( \le \bub \) called \df{traps}.
For trap \( \Rect(a, b) \), we will say that it
\df{starts} at its lower left corner \( a \).
 \end{definition}

 \begin{definition}[Wall values]\label{def:walls}
To describe the process \( \cW \), we introduce the concept of a
 \df{wall value} \( E = (B, \r) \).
Here \( B \) is the \df{body} which is a right-closed interval, and \df{rank} 
 \[
  \r \ge \R.
 \]
We write \( \Body(E) = B \), \( |E| = |B| \).
We will sometimes denote the body also by \( E \).
Let \( \Wvalues \) denote the set of all possible wall values.
 \end{definition}
Let
 \begin{align*}
  \bbZ_{+}^{(2)}
 \end{align*}
denote the set of pairs \( \pair{u}{v} \) with \( u<v \), \( u,v\in\bbZ_{+} \).
The random objects
 \begin{equation*}
  \begin{split}
     \cW_{\d}\sbsq\cB_{\d}  &\sbsq \Wvalues,
\\   \cS_{\d}\sbsq\cC_{\d} &\sbsq \bbZ_{+}^{(2)} \times \{-1, 1\}
                                           \txt{ for } \d = x,y,
\\   \cS_{2} &\sbsq \bbZ_{+}^{(2)}\times\bbZ_{+}^{(2)} \times \{-1, 1\} \times \{0,1,2\}
  \end{split}
 \end{equation*}
are also functions of \( X,Y \).
(Note that we do not have any \( \cC_{2} \).)

 \begin{definition}[Barriers and walls]\label{def:barrier}
The elements of \( \cW_{x} \) and \( \cB_{x} \) are called \df{walls} and
\df{barriers} of \( X \) respectively, where the sets
\( \cW_{x},\cB_{x} \) are functions of \( X \).
(Similarly for \( \cW_{y},\cB_{y} \) and \( Y \).)
In particular, elements of \( \cW_{x} \) are called \df{vertical walls}, and
elements of \( \cW_{y} \) are called \df{horizontal walls}.
Similarly for barriers.
When we say that a certain interval contains a wall or barrier we mean that
it contains its body.

A right-closed interval is called \df{external} if it intersects no walls.
A wall is called \df{dominant} if it is
surrounded by external intervals each of which is either of size \( \ge\bub \) or is
at the beginning of \( \bbZ_{+} \).
Note that if a wall is dominant then it contains every wall intersecting it.

For a vertical wall value $E=(B,\r)$ and a value of $X(B)$ making $E$ a
barrier of rank $\r$
we will say that $E$ is a \df{potential vertical wall} of rank $\r$
if there is an extension of $X(B)$ to a complete sequence $X$ that makes
$E$ a vertical wall of rank $\r$.
Similarly for horizontal walls.
\end{definition}

The last definition uses the fact following from
Condition~\ref{cond:distr}.\ref{i:distr.indep.barrier} that whether an interval $B$
is a barrier of the process $X$ depends only $X(B)$.

The set of barriers is a random subset of the set of all possible
wall values, and the set of walls is a random subset of the set of barriers.

 \begin{condition}
The parameter \( \bub \) is an upper bound on the size of 
every trap and the thickness of any barrier.
 \end{condition}

\subsubsection{Cleanness}

The set \( \cC_{x} \) is a function of the process \( X \), and the set
\( \cC_{y} \) is a function of the process \( Y \): they are
used to formalize (encode) the notions of cleanness given descriptive names below.

 \begin{definition}[One-dimensional cleanness]\label{def:1dim-clean}
For an interval \( I = \rint{a}{b} \) or \( I=\clint{a}{b} \),
if \( (a, b, -1) \in \cC_{x} \) then we say that point \( b \) of
\( \bbZ_{+} \) is \df{clean} in \( I \) for the sequence \( X \).
If \( (a,b,1) \in \cC_{x} \) then we say that point \( a \) is clean in \( I \).
From now on, whenever we talk about cleanness of an element of \( \bbZ_{+} \),
it is always understood with respect to either for the sequence \( X \) or for \( Y \).
For simplicity, let us just talk about cleanness, and so on, with
respect to the sequence \( X \).
A point \( x \in \bbZ_{+} \) is called left-clean (right-clean)
if it is clean in all intervals of the form \( \rint{a}{x} \), \( \clint{a}{x} \) (all intervals
of the form \( \rint{x}{b} \), \( \clint{x}{b} \)).
It is \df{clean} if it is both left- and right-clean.
If both ends of an interval \( I \) are clean in \( I \) 
then we say \( I \) is \df{inner clean}.

To every notion of one-dimensional
cleanness there is a corresponding notion of \df{strong
cleanness}, defined with the help of the process \( \cS \) in place of the
process \( \cC \).
 \end{definition}

Figure 8 of~\cite{GacsWalks11} illustrates one-dimensional cleanness.

 \begin{definition}[Trap-cleanness]\label{def:trap-clean}
For points \( u=\pair{u_{0}}{u_{1}} \), \( v=\pair{v_{0}}{v_{1}} \), 
\( Q = \Rect^{\eps}(u,v) \)
where \( \eps = \rightarrow \) or \( \uparrow \) or nothing,
we say that point \( u \) is \df{trap-clean in} \( Q \)
(with respect to the pair of sequences \( \tup{X,Y} \))
if \( (u, v, 1, \eps') \in \cS_{2} \), where \( \eps' = 0,1,2 \) depending on
whether \( \eps = \rightarrow \) or \( \uparrow \) or nothing.
Similarly, point \( v \) is \df{trap-clean in} \( Q \)
if \( (u, v, -1, \eps') \in \cS_{2} \).
It is  \df{upper right trap-clean}, if it is trap-clean in the lower left 
corner of all rectangles.
It is  \df{trap-clean}, if it is trap-clean in all rectangles.
 \end{definition}

 \begin{definition}[Complex two-dimensional sorts of cleanness]\label{def:H-clean}
We say that point \( u \) is \df{clean} in \( Q \) when
it is trap-clean in \( Q \) and its projections are clean in the
corresponding projections of \( Q \).

If \( u \) is clean in all such left-open rectangles
then it is called \df{upper right rightward clean}.
We delete the ``rightward'' qualifier here
if we have closed rectangles in the definition here instead of
left-open ones.
Cleanness with qualifier ``upward'' is defined similarly.
Cleanness of \( v \) in \( Q \) and lower left cleanness of \( v \) are
defined similarly, using \( (u,v, -1, \eps') \),
except that the qualifier is unnecessary: all our rectangles are upper
right closed.

A point is called \df{clean} if it is upper right clean and lower left clean.
If both the lower left and upper right 
points of a rectangle \( Q \) are clean in \( Q \) then \( Q \) is called \df{inner clean}.
If the lower left endpoint is lower left clean and the upper
right endpoint is upper right rightward clean 
then \( Q \) is called \df{outer rightward clean}.
Similarly for \df{outer upward clean} and \df{outer-clean}.

We will also use a \df{partial} versions of cleanness.
If point \( u \) is trap-clean in \( Q \) and its projection on the \( x \) axis 
is \emph{strongly} clean in the same projection of \( Q \)
then we will say that \( u \) is \df{H-clean} in \( Q \).
Clearly, if \( u \) is H-clean in \( Q \) and its projection on the \( y \) axis is
clean in (the projection of) \( Q \) then it is clean in \( Q \).
We will call rectangle \( Q \) \df{inner} H-clean if both its lower left and upper
right corners are H-clean in it.
It is now clear what is meant for example by a point 
being \df{upper right rightward H-clean}.

The notion \df{V-clean} is defined similarly when we interchange horizontal and
vertical.
 \end{definition}

Figure 9 of~\cite{GacsWalks11} illustrates 2-dimensional cleanness.

\subsubsection{Hops}
Hops are intervals and rectangles for which we will be able to
give some guarantees that they can be passed.

\begin{definition}[Hops]\label{def:hop}
A right-closed horizontal interval \( I \) is called a \df{hop}
if it is inner clean and contains no vertical wall.
A closed interval \( \clint{a}{b} \) is a hop if \( \rint{a}{b} \) is a hop.
Vertical hops are defined similarly.

We call a rectangle \( I\times J \)
a \df{hop} if it is inner clean and contains no trap, and no wall (in either of its projections).
\end{definition}

 \begin{remarks}
   \begin{enumerate}[1.]
   \item 
An interval or rectangle that is a hop can be empty: this is the case if
the interval is \( \rint{a}{a} \), or the rectangle is, say,
\( \Rect^{\rightarrow}(u,u) \).
  \item The slight redundancy of considering separately
\( \R^{\uparrow} \) and \( R^{\rightarrow} \) in the present paper is there just 
for the sake of some continuity with~\cite{GacsWalks11}.
The present paper could just use rectangles that are both bottom-open and
left-open.
On the other hand, \cite{GacsWalks11} 
started from a graph \( \cG \) with only horizontal and vertical edges.
Then the bottom left point of a rectangle that is both bottom-open and left-open
would be cut off completely.
   \end{enumerate}
 \end{remarks}

 \begin{definition}[Sequences of walls]\label{def:neighbor-seq}
Two disjoint walls are called \df{neighbors} if the interval
between them is a hop.
A sequence \( W_{i} \in \cW \) of walls
\( i=1,2,\dots, n \) along with the intervals \( I_{1},\dots,I_{n-1} \) between them 
is called a \df{sequence of neighbor walls}
if for all \( i > 1 \), \( W_{i} \) is a right neighbor of \( W_{i-1} \).
We say that an interval \( I \) is \df{spanned} by the sequence of neighbor
walls \( W_{1},W_{2},\dots,W_{n} \) if 
\( I=W_{1}\cup I_{1}\cup W_{2}\cup\dots\cup W_{n} \).
We will also say that \( I \) is spanned by the sequence \( \tup{W_{1},W_{2},\dots} \)
if both \( I \) and the sequence are infinite and \( I=W_{1}\cup I_{1}\cup
W_{2}\cup\dots \).
If there is a hop \( I_{0} \) adjacent on the left to \( W_{1} \) and a hop
\( I_{n} \) adjacent on the right to \( W_{n} \) (or the sequence \( W_{i} \) is
infinite) then this system is called an
\df{extended sequence of neighbor walls}.
We say that an interval \( I \) is \df{spanned} by this extended sequence
if \( I=I_{0}\cup W_{1}\cup I_{1}\cup \dots\cup I_{n} \) (and correspondingly
for the infinite case).
 \end{definition}

\subsubsection{Holes}\label{sss.holes}

 \begin{definition}[Reachability]\label{def:reachability}
We say that point \( v \) is \df{reachable} from point \( u \) in \( \cM \) 
(and write \( u \leadsto v \)) if it is reachable in the graph \( \cG \).
 \end{definition}
  
 \begin{remark}
Point \( u \) itself may be closed even if \( v \) is reachable from \( u \).
 \end{remark}

 \begin{definition}[Slope conditions]
We will say that points \( u=\pair{u_{0}}{u_{1}} \) and \( v=\pair{v_{0}}{v_{1}} \)
with \( u_{\d}<v_{\d} \), \( \d=0,1 \) satisfy the \df{slope conditions} if there is a
(non-integer) point \( v'=\pair{v'_{0}}{v'_{1}} \) with \( 0\le v_{0}-v'_{0},v_{1}-v'_{1}<1 \) 
such that
 \begin{align*}
   \slb_{x}\le \slope(u,v')\le \slb_{y}^{-1}.
 \end{align*}
 \end{definition}
The simple slope conditions would be \( \slb_{x}\le \slope(u,v)\le \slb_{y}^{-1} \),
but we are a little more lenient, to allow for some rounding.

Intuitively, a hole is a place at which we can pass through a wall.
We will also need some guarantees of being able to reach the hole and being
able to leave it.

 \begin{definition}[Holes]\label{def:holes}
Let \( a=\pair{a_{0}}{a_{1}} \), \( b=\pair{b_{0}}{b_{1}} \), be a pair of points, and
let the interval \( I = \rint{a_{1}}{b_{1}} \) be the body of
a horizontal barrier \( B \).
For an interval \( J = \rint{a_{0}}{b_{0}} \) 
we say that \( J \) is a vertical \df{hole passing through} \( B \), 
or \df{fitting}  \( B \), if \( a\leadsto b \) within the rectangle
\( J\times\clint{a_{1}}{b_{1}} \).
For technical convenience, we also require \( |J|\le\slblb^{-1}|I| \).
Consider a point \( \pair{u_{0}}{u_{1}} \) with \( u_{i}\le a_{i} \), \( i=0,1 \).
The hole \( J \) is called \df{good as seen from} point \( u \) if \( a \) is H-clean in
\( \Rect^{\to}(u,a) \), and \( b \) is upper-right rightward
H-clean (recall Definition~\ref{def:H-clean}). 
It is \df{good} if it is good as seen from any such point \( u \).
Note that this way the horizontal cleanness is required to be strong, but no
vertical cleanness is required (since the barrier \( B \) was not required to
be outer clean).

Each hole is called \df{lower left clean}, upper right clean, and so on,
if the corresponding rectangle is.

Horizontal holes are defined similarly.
 \end{definition}

The conditions defining the graph \( \cG \) imply that the slope of any path is
between \( \slblb \) and \( 1 \).
It follows that 
the width of a horizontal hole is at most \( \bub \), and the width of
a vertical hole is at most \( \slblb^{-1}\bub \).

\subsection{Conditions on the random process}

The conditions will depend on a constant
 \begin{equation}\label{eq:hxp}
   \hxp = 0.015
 \end{equation}
whose role will become clear soon, and on 
 \begin{equation}\label{eq:lg-def}
   \lg = 2^{1/2}.
 \end{equation}

 \begin{definition}
The function 
 \begin{equation*}
  \p(\r,l)
 \end{equation*}
is defined as the supremum of probabilities (over all points \( t \))
that any vertical or horizontal barrier with rank \( \r \) and size \( l \) starts at \( t \).
 \end{definition}

 \begin{remark}
In the probability bounds of the paper~\cite{GacsWalks11} we also conditioned on
arbitrary starting values in an interval, since there the processes \( X,Y \)
were Markov chains, not necessarily Bernoulli.
We omit this conditioning in the interest of readability, 
as it is not needed for the present application.
Technically speaking, in this sense the mazery defined here is not a
generalization of the earlier one.
 \end{remark}

We will use some additional constants,
 \begin{align}\label{eq:aux}
 \aux_{1}=2,\;\aux_{2},\;\aux_{3}, 
 \end{align}
some of which will be chosen later.

 \begin{definition}[Probability bounds]
Let
 \begin{align}\label{eq:wall-prob}
  \p(\r) &= \aux_{2} \r^{-\aux_{1}} \lg^{-\r},
\\ \label{eq:h-def}
         \h(\r)  &= \aux_{3}\lg^{-\hxp\r}.
 \end{align}
 \end{definition}

 \begin{condition}\label{cond:distr}\
 \begin{enumerate}[\upshape 1.]

  \item\label{i:distr.indep}(Dependencies)
  \begin{enumerate}[\upshape a.]

    \item\label{i:distr.indep.trap}
For any rectangle \( I \times J \), the event that it is a
trap is a function of the pair \( X(I), Y(J) \).

    \item\label{i:distr.indep.barrier}
For a vertical wall value \( E \) the event \( \setof{E\in\cB} \) (that is the event
that it is a vertical barrier) is a function of \( X(\Body(E)) \).

Similarly for horizontal barriers.

    \item\label{i:distr.indep.left-clean} 
For integers \( a < b \), and the events defining strong horizontal cleanness, that
is \( \setof{(a,b,-1) \in \cS_{x}} \) and \( \setof{(a,b,1) \in \cS_{x}} \),
are functions of \( X(\rint{a}{b}) \).
Similarly for vertical cleanness and the sequence \( Y \).

When \( X,Y \) are fixed, then for a fixed \( a \), the
(strong and not strong) cleanness of \( a \) in \( \rint{a}{b} \) 
is decreasing as a function of \( b-a \), 
and for a fixed \( b \), the (strong and not strong) cleanness 
of \( b \) in \( \rint{a}{b} \) is decreasing as a function of \( b-a \). 
These functions reach their minimum at \( b-a=\bub \): thus, for example if
\( x \) is (strongly or not strongly) 
left clean in \( \rint{x-\bub}{x} \) then it is (strongly or not strongly) 
left clean.

    \item\label{i:distr.indep.ur-clean}
For any rectangle \( Q=I \times J \),
the event that its lower left corner is trap-clean in \( Q \),
is a function of the pair \( X(I), Y(J) \).

Among rectangles with a fixed lower left corner, the
event that this corner is
trap-clean in \( Q \) is a decreasing function of \( Q \)
(in the set of rectangles partially ordered by containment).
In particular, the trap-cleanness of \( u \) in \( \Rect(u,v) \) implies its 
trap-cleanness
in \( \Rect^{\rightarrow}(u,v) \) and in \( \Rect^{\uparrow}(u,v) \).
If \( u \) is upper right trap-clean in the left-open or bottom-open or
closed square of size \( \bub \),
then it is upper right trap-clean in all rectangles \( Q \) of the same
type. 
Similar statements hold if we replace upper right with lower left.

  \end{enumerate}

Whether a certain wall value \( E=(B,\r) \) is a vertical barrier depends only on
\( X(B) \).
Whether it is a vertical 
wall depends also only on \( X \)---however, it may depend on the values of \( X \)
outside \( B \).
Similarly, whether a certain horizontal interval is inner clean depends only the
sequence \( X \) but may depend on the elements outside it,
but whether it is strongly inner clean depends only on \( X \) inside the interval.

Similar remarks apply to horizontal wall values and vertical cleanness with the
process \( Y \).

  \item\label{i:distr.combinat}(Combinatorial requirements)
   \begin{enumerate}[\upshape a.]
  
    \item\label{i:distr.inner-clean}
A maximal external interval  (see
Definition~\ref{def:barrier}) of size \( \ge\bub \) or one starting at \( -1 \)
is inner clean.

    \item\label{i:distr.cover}
An interval \( I \) that is surrounded by maximal external intervals of size
\( \ge\bub \) is spanned by a
sequence of (vertical) neighbor walls (see Definition~\ref{def:neighbor-seq}).
This is true even in the case when \( I \) starts at 0 and even if it is infinite.
To accommodate these cases, we require the following, which is somewhat harder to
parse:
Suppose that interval \( I \) is adjacent on the left to a maximal external interval
that either starts at \( -1 \) or has size \( \ge \bub \).
Suppose also that it is either adjacent on the right to a similar
interval or is infinite.
Then it is spanned by a (finite or infinite) sequence of neighbor walls.
In particular, the whole line is spanned by an extended sequence of
neighbor walls.

    \item\label{i:distr.clean.1}
If a (not necessarily integer aligned)
right-closed interval of size \( \ge 3 \bub \) contains no wall,
then its middle third contains a clean point.

    \item\label{i:distr.clean.2}
Suppose that a rectangle \( I\times J \) with (not necessarily integer aligned)
right-closed \( I,J \) with \( |I|, |J|\ge 3\bub \) contains
no horizontal wall and no trap, and \( a \) is a right 
clean point in the middle third of \( I \).
There is an integer \( b \) in the middle third of \( J \) such that the point
\( \pair{a}{b} \) is upper right clean.
A similar statement holds if we replace upper right with
lower left (and right with left).
Also, if \( a \) is clean then we can find a point \( b \) 
in the middle third of \( J \) such that \( \pair{a}{b} \) is clean.

There is also a similar set of statements if we vary \( a \) instead of \( b \).

\item\label{i:distr.reachable}(Reachability)
If points \( u, v \) satisfying the slope conditions
are the starting and endpoint of a rectangle 
that is a hop, then \( u \leadsto v \).
The rectangle in question is allowed to be
bottom-open or left-open, but not both.

(In the present paper, we could even allow the rectangle to be both bottom open and
left open,
since the graph \( \cG \) has no horizontal and vertical edges anyway.
But we will not use this.)
   \end{enumerate}

   \item\label{i:distr.bounds}(Probability bounds)
  \begin{enumerate}[\upshape a.]

   \item\label{i:distr.trap-ub}
Given a string \( x = \tup{x(0), x(1), \dotsc} \), a point \( \pair{a}{b} \),
let \( \cF \) be the event that a trap starts at \( \pair{a}{b} \).
We have 
 \begin{equation*}
     \Prob\paren{\cF \mid X = x} \le \tub.
 \end{equation*}
 The same is required if we exchange \( X \) and \( Y \).

   \item\label{i:distr.wall-ub}
We have \( \p(\r) \ge \sum_{l} \p(\r,l) \).

  \item\label{i:distr.ncln}
We require that for all \( a < b \) and all \( u=\pair{u_{0}}{u_{1}} \),
\( v=\pair{v_{0}}{v_{1}} \)
\begin{align}\label{eq:ncln.1dim}
\Pbof{ a \txt{ (resp. \( b \)) } \txt{ is not strongly clean in } \rint{a}{b}}
&\le \ncln_{\w},
\end{align}
Further, for \( Q=\Rect^{\rightarrow}(u,v) \) or \( \Rect^{\uparrow}(u,v) \) or
\( \Rect(u,v) \), for all sequences \( y \)
\begin{align*}
\Pbof{ u \txt{ is not trap-clean in } Q  \mid  Y=y}&\le\ncln_{\w},
\\
\Pbof{ v \txt{ is not trap-clean in }Q \mid Y=y} &\le\ncln_{\t}
 \end{align*}
and similarly with \( X \) and \( Y \) reversed.

   \item\label{i:distr.hole-lb}
Let \( u \le v < w \), and \( a \) be given with \( v-u \le \slblb^{-2}\bub \),
and define
 \begin{align*}
                      b &= a + \cei{\slb_{y}(v-u)},
\\                  c &= b\lor (a+\flo{\slb_{x}^{-1}(v-u)}).
 \end{align*}
Assume that \( Y=y \) is fixed in such a way that
\( B \) is a horizontal wall of rank \( \r \) with body \( \rint{v}{w} \).
For a \( d \in \clint{b}{c} \) let 
\( Q(d)=\Rect^{\rightarrow}(\pair{a}{u}, \pair{d}{v}) \).
Let
\begin{align*}
   F(u, v;\, a, d)
 \end{align*}
be the event (a function of \( X \)) that 
\( Q(d) \) contains no traps or vertical barriers, and is inner H-clean.
Let
\begin{align*}
   E = E(u, v, w;\, a)
 \end{align*}
be the event that at some point \( d\in\clint{b}{c} \)
a vertical hole fitting \( B \) starts, and event \( F(u, v;\, a, d) \) holds.
Then
 \begin{equation*}
   \Prob\paren{E \mid Y = y} \ge (v-u+1)^{\hxp} \h(\r).
 \end{equation*}
The same is required if we exchange horizontal and vertical, \( X \) with \( Y \), further
\( \slb_{y} \) with \( \slb_{x} \), and 
define \( Q(d)=\Rect^{\uparrow}(\pair{u}{a}, \pair{v}{d}) \).
  \end{enumerate}
Figure 10 of~\cite{GacsWalks11} illustrates the last condition.
 \end{enumerate}
 \end{condition}

The following lemma shows how the above condition will serve for
passing from point \( \pair{a}{u} \) past the wall.

\begin{lemma}\label{lem:pre-hole}
In Condition \ref{cond:distr}.\ref{i:distr.hole-lb}, the points
\( \pair{a}{u} \), \( \pair{d}{v} \) always satisfy the slope conditions.
 \end{lemma}
 \begin{proof}
Consider the case of horizontal walls.
We have \( b-a\ge\slb_{y}(v-u) \) by definition.
If \( c>b \) then also \( d-a\le \slb_{x}^{-1}(v-u) \) by definition 
for any \( d\in\clint{b}{c} \).
Assume therefore \( c=b \), then \( d=b=c \).
We claim that the points \( \pair{a}{u} \) and \( \pair{b}{v} \) satisfy
the slope conditions.
Indeed, set  \( b'=a+\slb_{y}(v-u) \), then \( b-1<b'\le b \), and
\begin{align*}
    1/\slb_{y} = \slope(\pair{a}{u},\pair{b'}{v}) > \slb_{x}
 \end{align*}
since \( \slb_{x}\slb_{y}<1 \).
The case for vertical walls is similar.
 \end{proof}

\begin{remarks}\label{rem:distr}\  
 \begin{enumerate}[\upshape 1.]

  \item\label{i:middle-third}
 Conditions~\ref{cond:distr}.\ref{i:distr.clean.1} 
and~\ref{cond:distr}.\ref{i:distr.clean.2} imply the following.
Suppose that a right-upper closed square \( Q \) of size \( 3 \bub \) contains
no wall or trap.
Then its middle third contains a clean point.

\item Note the following asymmetry: the probability bound on the upper right
  corner of a rectangle 
  not being trap-clean in it is \( \ncln_{\t} \) which is bounded only by \( 0.55 \), while 
the bound of the lower left corner not being trap-clean in it is \( \ncln_{\w} \),
which is bounded by \( 0.05 \).

 \item\label{i:rem.reachability}
With respect to condition~\ref{cond:distr}.\ref{i:distr.reachable} 
note that not all individual edges satisfy the slope condition; indeed, some
arguments will make use of this fact.

  \item\label{i:rem.hole-lb}
The most important special case of
Condition \ref{cond:distr}.\ref{i:distr.hole-lb} is
\( v = u \), then it says 
that for any horizontal wall \( B \) of rank \( \r \), at any point \( a \),
the probability that there is a vertical hole passing through \( B \) at point
\( a \) is at least \( \h(\r) \).
 \end{enumerate}
\end{remarks}

\subsection{Base mazery}

Let us define a mazery \( \cM^{1} \) corresponding to the embedding problem.

 \begin{example}[Embedding mazery]\label{xmp:base}
Let
\begin{align*}
  \slblb &=\slb_{x1} = 1/2\m,\quad \slb_{y1}=\m,\quad \R_{1}=2\m, 
\\ \bub_{1}&=\lg^{\bubxp\R_{1}},
\\ \ncln_{\w 1}&=0,\quad \ncln_{\t 1}=0.5, \quad \tub_{1}=0, 
 \end{align*}
\( \bubxp=0.15 \) (the choice will be justified in Section~\ref{sec:params}).

Let \( \cG(X,Y)=\cG_{3\m}(X,Y) \) be the graph defined in the introduction.
Let \( \cT=\emptyset \), that is there are no traps.

An interval \( \rint{i}{i+l} \) is a vertical barrier and wall if and only if
\( \m\le l < 2\m \), and 
\( X(i+1)=X(i+2)=\dots=X(i+l) \).
Similarly, it is a horizontal barrier and wall if and only if
\( Y(i+1)=Y(i+2)=\dots=Y(i+l) \).
We define the common rank of these barriers to be \( \R_{1} \).

Every point is strongly clean in all one-dimensional senses.
All points are upper right trap clean.
A point \( \pair{i}{j} \) is lower left trap-clean if \( X(i)=Y(j) \).
On the other, hand if \( X(i)\ne Y(j) \) then it is not trap-clean in any 
nonempty rectangles whose upper right corner it is.

Note that even though the size of the largest walls or traps is bounded by \( \m \), 
the bound \( \bub_{1} \) is defined to be exponential in \( \m \).
This will fit into the scheme of later definitions. 
 \end{example}

 \begin{lemma}\label{lem:base}
The definition given in Example~\ref{xmp:base} satisfies the mazery
conditions, for sufficiently large \( \R_{1}(=2\m) \).
 \end{lemma}
 \begin{proof}
We will write \( \R=\R_{1} \) throughout the proof.
   \begin{enumerate}[1.]
   \item 
Almost all combinatorial and dependency conditions are satisfied trivially; here are
the exceptions.
Condition \ref{cond:distr}.\ref{i:distr.cover} says that
an interval \( I \) surrounded by maximal external intervals of size
\( \ge\bub \) is spanned by a sequence of (vertical) neighbor walls.
Since \( I \) is a surrounded by maximal external intervals, there is a wall 
of size \( \m \) at the beginning of \( I \) and one of size \( \m \)
at the end of \( I \).
If \( |I|<2\m \) then \( I \) is itself a wall.
Otherwise, we start with the wall \( J_{1} \) of size \( \m \) at the beginning,
and build a 
sequence of disjoint walls \( J_{1}, J_{2},\dots \) of size \( \m \)
recursively with each \( J_{i}  \) at a distance \( \ge\m \) from the right end
of \( I \).
The next wall is chosen always to be the closest possible satisfying these conditions.
Finally, we add the wall of size \( \m \) at the end of \( I \).
Since every point is by definition strongly clean in all one-dimensional
senses, the sequence we built is a spanning sequence of neighbor walls.

In Condition \ref{cond:distr}.\ref{i:distr.clean.2},
only lower left cleanness is not automatic.
Suppose that a rectangle \( I\times J \) with
right-closed \( I,J \) with \( |I|, |J|\ge 3\bub_{1} \) contains
no horizontal wall and no trap, and \( a \) is a 
point in the middle third of \( I \).
We must show that there is an integer \( b \) in the middle 
third of \( J \) such that the point \( \pair{a}{b} \) is lower left 
clean.
This condition would now only be violated
if \( Y(b)\ne X(a) \) for all \( b \) in the middle third.
But since \( \bub_{1}>\m \), this would create a horizontal wall, 
which was excluded.
The same argument applies if we vary \( a \) instead of \( b \).

\item Let us verify the reachability condition.
Let \( u<v \), \( v=\pair{v_{0}}{v_{1}} \)
be points with the property that there is a \( v'=\pair{v'_{0}}{v'_{1}} \)
with \( 0\le v_{d}-v'_{d}<1 \) for \( d=0,1 \), and
\( \slope(u, v') \ge \slb_{x1}=1/2\m \), \( 1/\slope(u, v') \ge \slb_{y1}=\m \).
If they are the starting and endpoint of a (bottom-open or left-open)
rectangle that is a hop, then the condition requires \( u \leadsto v \).
The hop property implies \( X(v_{0})=Y(v_{1}) \): indeed, otherwise
the rectangle would not be inner clean.

Without loss of generality, let \( u=\pair{0}{0} \), \( v=\pair{a}{b} \),
\( v'=\pair{a'}{b'} \).
Now the slope requirements mean \( \m \le a'/b' \le 2\m \), hence
\( \m < a/(b-1) \),  \(  (a-1)/b < 2\m \), and so
 \begin{align*}
   \m(b-1)< a \le 2\m b.
 \end{align*}
It is then easy to see that we can choose a sequence
 \begin{align*}
    0\le  s_{1}<s_{2}<\dots <s_{b-1}<a
 \end{align*}
with the properties
 \begin{align*}
             s_{1}&\le 2\m,
\\  s_{i}+\m &\le s_{i+1} \le s_{i}+2\m,
\\  s_{b-1}+\m & < a \le s_{b-1}+2\m.   
 \end{align*}
Indeed, if the \( s_{i} \) are all made minimal
then \( s_{b-1}+\m= \m(b-2)+\m < a \).
On the other hand, if all these values are maximal then
\( s_{b-1}+2\m = 2\m(b-1)+2\m \ge a \).
Choosing the values in between we can satisfy both inequalities.

The hop requirement implies that there is no vertical wall in
\( \rint{0}{a} \), that is
there are no \( \m \) consecutive numbers in this interval
with identical values of \( X(i) \).
It also implies \( X(a)=Y(b) \).
Let us choose \( a_{i} \) from the interval \( \rint{s_{i}}{s_{i}+\m} \) such
that \( X(a_{j})=Y(j) \).
By construction we have \( 0<a_{1}\le 2\m \), \( 0<a_{i+1}-a_{i} < 3\m \),
\( 0<a-a_{b-1}\le 2\m \).
Thus the points \( \pair{a_{j}}{j} \) form a path in the graph \( \cG=\cG_{3\m}(X,Y) \)
from \( \pair{0}{0} \) to \( v \).

\item
Since there are no traps, the trap probability upper bound is satisfied
trivially.

\item
Consider the probability bounds for barriers.
Since the rank is the same for both horizontal and vertical barriers, it is
sufficient to consider vertical ones.
Clearly \( \p(\r,l)=0 \) unless \( \r=\R \), \( l\ge\m \), in which case it is 
\( 2^{-l} \); hence \( \sum_{l} \p(\R,l)\le 2^{-\m+1} \).
For \( \p(\R) \ge \sum_{l} \p(\R,l) \), we need:
 \begin{align*}
     2^{-\m+1}&\le \aux_{2} \R^{-\aux_{1}} \lg^{-\R},
 \end{align*}
which holds for \( \R \) sufficiently large, since \( \lg=2^{1/2} \)
by~\eqref{eq:lg-def}. 

\item Consider the bounds~\ref{cond:distr}.\ref{i:distr.ncln} on the probability
that some point in not clean in some way.
Only the lower left trap-cleanness is now in question, so only the bound
\begin{align*}
\Pbof{ v \txt{ is not trap-clean in }Q \mid Y=y} &\le\ncln_{\t}
 \end{align*}
must be checked.
The event happens here only if \( X(v_{0})\ne y(v_{1}) \): its probability, \( \half \), is
now equal to \( \ncln_{\t} \) by definition, so the inequality holds.
The argument is the same when horizontal and vertical are exchanged.

\item
Consider Condition \ref{cond:distr}.\ref{i:distr.hole-lb} for 
a vertical wall, with the parameters \( a,u,v,w \).
With our parameters, it gives \( b = a + \cei{(v-u)/2\m} \), and
events \( F(u, v;\, a,d) \) and \( E = E(u, v, w;\, a) \).
By definition of cleanness now, the lower left corner of any rectangle is
automatically V-clean in it.
The requirement is
 \begin{equation*}
   \Prob\paren{E \mid X = x} \ge (v-u+1)^{\hxp} \h(\r).
 \end{equation*}
Since \( B \) is a wall we have \( X(v+1)=\dots=X(w) \).

Let \( A \) denote
the event that interval \( \rint{a}{b} \) contains a horizontal barrier.
Then the probability of \( A \) is bounded by \( 1/8 \) if \( \R=2\m \) is sufficiently large.
Indeed, 
\begin{align*}
 b-a\le \slblb^{-2}\bub = 4\m^{2}\lg^{2\bubxp\m},
  \end{align*}
while the probability of a barrier at a point is \( \le 2^{-\m} \).
Via the union bound, we bound the probability by the product of these two numbers.

Let \( E' \) be the event that \( Y(b+1)=X(w) \), further \( b>a\imp Y(b)=X(v) \).
It has probability at least \( \fourth \).
This event implies that
\( \pair{v}{b} \) is trap-clean in \( Q(b) \), so \( Q(b) \) becomes inner V-clean.
It also implies a horizontal hole \( \rint{b}{b+1} \) fitting the
wall \( B \), as we can simply go from \( \pair{v}{b} \) to \( \pair{w}{b+1} \) on an edge
of the graph \( \cG \).

Lemma~\ref{lem:pre-hole} implies that \( Q(b) \) satisfies the slope conditions, so
\( E'\setminus A \) implies also event \( F(u,v;\, a,b) \), so also event \( E(u,v,w;\, a) \).
So \( \fourth-\frac{1}{8} \) lowerbounds the probability of event
\( E' \setminus A \sbsq E \).
It is sufficient to lowerbound therefore \( \frac{1}{8} \) by
\begin{align*}
 (v-u+1)^{\hxp} \h(\R)=(v-u+1)^{\hxp}\lg^{-\hxp\R}.
 \end{align*}
Using the value \( \bub=\lg^{\bubxp\R} \) and the bound 
\( v-u\le\slblb^{-2}\bub \), it is sufficient to have
 \begin{align*}
      \aux_{3}(2\slblb^{-2}\bub)^{\hxp}\lg^{-\hxp\R_{1}}  
      &=\aux_{3}(8\R^{2})^{\hxp}\lg^{-\hxp\R(1-\bubxp)} 
      \le 1/8,
 \end{align*}
which is true with \( \R \) sufficiently large.

\item
Consider now the probability lower bound on passing a horizontal wall
of size \( l \), where \( \m\le l < 2\m \),
that is Condition~\ref{cond:distr}.\ref{i:distr.hole-lb}.
This condition, for our parameters, defines
\( b= a + \m(v-u) \).
It assumes that \( Y=y \) is fixed in such a way that 
\( B \) is a horizontal wall of rank \( \r \) with body \( \rint{v}{w} \).
The requirement is
\( \Prob\paren{E \mid Y = y} \ge (v-u+1)^{\hxp} \h(\r) \).
Now, since \( B \) is a wall we have \( Y(v+1)=\dots=Y(w) \).

Let \( A_{1} \) denote the event that there is a vertical wall in \( \rint{a}{b} \).
Let \( E' \) be the event that \( b>a\imp X(b)=Y(v) \).
It implies that \( Q(b) \) is inner H-clean.
The event \( E'\setminus A_{1} \) implies \( F(u,v;\, a,b) \).

Let \( A_{2} \) denote the event that there is an interval 
\( I\sbsq \rint{b}{b + l \m} \) of size \( \m \) with \( X(i)\ne Y(w) \) 
for all \( i\in I \).
Let \( E'' \) denote the event \( X(b+l\m)=Y(w) \).
Then \( E''\setminus A_{2} \) implies that 
\( \rint{b}{b+l\m} \) is a vertical hole fitting the wall \( B \).
Indeed, just as in the proof of the reachability condition, already the fact
that there is no interval \( I\sbsq \rint{b}{b+l\m} \) of size \( \m \) with 
\( X(i)\ne Y(w) \) for all \( i \), and that the pair
of points \( \pair{b}{v} \), \( \pair{b+l\m}{w} \) satisfies the slope conditions,
implies that the second point is reachable from the first one.

So the event \( E'\cap E''\setminus(A_{1}\cup A_{2}) \) implies
\( E(u,v,w;\, a) \).
Let us upperbound the probability that this does not occur.
Since events \( E',E'' \) are independent, the probability that \( E'\cap E'' \)
does not occur is at most \( \frac{3}{4} \).
The probability of \( A_{1}\cup A_{2} \) can be bounded by 
\( \frac{1}{8} \), just as in the case of passing a horizontal wall.
Thus we found \( \Prob(E)\ge 1-\frac{7}{8}=\frac{1}{8} \).
It is sufficient to lowerbounded this by \( (v-u+1)^{\hxp}\aux_{3}\lg^{-\hxp\R} \). 
So we will be done if
 \begin{align*}
 (2\slblb^{-2}\bub)^{\hxp}\aux_{3}\lg^{-\hxp\R} 
  =  (8\R^{2})^{\hxp}\aux_{3}\lg^{-\hxp\R(1-\bubxp)}\le 1/8,
 \end{align*}
which holds if \( \R \) is sufficiently large.
 \end{enumerate}
 \end{proof}

\section{Application to the theorem}\label{sec:main-lemma}

Theorem~\ref{thm:main} follows from Lemma~\ref{lem:base} and the following
theorem:

\begin{theorem}\label{thm:mazery}
In every mazery with a sufficiently large rank lower bound 
there is an infinite path starting from the origin,
with positive probability.
\end{theorem}

The proof will use the following definitions.

\begin{definition}
In a mazery \( \cM \), let \( \cQ \) be the event that the origin
 \( \pair{0}{0} \) is not upper right clean,
and \( \cF(n) \) the event that the square \( \clint{0}{n}^{2} \) contains some
wall or trap.
\end{definition}

 \begin{lemma}[Main]\label{lem:main}
Let \( \cM^{1} \) be a mazery.
If its rank lower bound
is sufficiently large then a sequence of mazeries \( \cM^{k} \), \( k>1 \) can
be constructed on a common probability space, sharing the graph \( \cG \) 
of \( \cM^{1} \) and the parameter \( \slblb \), and satisfying 
\begin{align}
\nonumber     \slb_{j, k+1} &\ge\slb_{j,k} \txt{ for } j=x,y ,
\\\nonumber \bub_{k}/\bub_{k+1} &< \slblb^{2}/2,
\\\label{eq:main.big-sum}
 1/4&>\sum_{k=1}^{\infty}\Prob\bigparen{\cF_{k}(\bub_{k+1})\cup(\cQ_{k+1}\setminus \cQ_{k})}.
 \end{align}
 \end{lemma}

Most of the paper will be taken up with the proof of this lemma.
Now we will use it to prove the theorem.

\begin{proof}[Proof of Theorem~\protect\ref{thm:mazery}]
Let \( u=\pair{0}{0} \) denote the origin.
The mazery conditions imply \( \Prob(\cQ_{1})\le 0.15 \).
Let us construct the series of mazeries \( \cM^{k} \) satisfying the conditions
of Lemma~\ref{lem:main}.
These conditions imply that the probability that one of the events
 \( \cF_{k}(\bub_{k+1}) \), \( \cQ_{k+1}\setminus \cQ_{k} \) hold is less than
 \( 0.25 \).
Hence the probability that
 \( \bigcup_{k=1}^{\infty}\cQ_{k}\cup\cF_{k}(\bub_{k+1}) \) holds is
 at most \( 0.4 \).
With probability at least \( 0.6 \) none of these events holds.
Assume now that this is the case.
We will show that there is an infinite number of points \( v \) of the graph
reachable from the origin.
The usual compactness argument implies then an
infinite path starting at the origin.

Under the assumption, 
in all mazeries \( \cM^{k} \) the origin \( u \) is upper right clean,
and the square \( \clint{0}{\bub_{k+1}}^{2} \) contains no walls or traps.
Let \( \slb_{x}=\slb_{x,k} \), \( \slb_{y}=\slb_{y,k} \).
Consider the point \(\pair{a}{b}=\pair{\bub_{k+1}}{\slb_{x}\bub_{k+1}} \).
Then the square \( \pair{a-3\bub_{k}}{b}+\clint{0}{3\bub_{k}}^{2} \) is inside
the square \( \clint{0}{\bub_{k+1}}^{2} \), and contains no walls or traps.
The mazery conditions imply that then its middle, the square
\( \pair{a-2\bub_{k}}{b+\bub_{k}}+\clint{0}{\bub_{k}}^{2} \) contains a clean
point \( v=\pair{v_{0}}{v_{1}} \).
By its construction, the rectangle \( \Rect(u,v) \) is a hop.
Let us show that 
it also satisfies the slope lower bounds of mazery \( \cM^{k} \), and therefore by
the reachability condition, \( u\leadsto v \).
Indeed, by its construction, \( v \) is above the
line of slope \( \slb_{x} \) starting from \( u \).
On the other hand, using the bound
on \( \bub_{k}/\bub_{k+1} \) of Lemma~\ref{lem:main} and \( \slblb\le 1/2 \):
\begin{align*}
 \slope(u,v)&=\frac{v_{1}}{v_{0}} 
  \le \frac{\slb_{x}\bub_{k+1}+2\bub_{k}}{\bub_{k+1}-2\bub_{k}}
  \le \frac{\slb_{x}+\slblb^{2}}{1-\slblb^{2}} 
\\ &\le \slb_{x}\frac{1+\slblb}{1-\slblb^{2}} =\frac{\slb_{x}}{1-\slblb} \le 1/\slb_{y}
 \end{align*}
by~\eqref{eq:slb-ub}.
\end{proof}

\begin{remark}\label{rem:pos-prob-small}
  It follows from the proof that if we use the base mazery of Example~\ref{xmp:base} then
the probability of the existence of an infinite path in Theorem~\ref{thm:mazery} 
converges to 1 as \( \m\to\infty \).
Indeed in this case \( \Prob(Q_{1})=0 \), and the sum in~\eqref{eq:main.big-sum} converges to 0
as \( \m\to\infty \).
\end{remark}

 \section{The scaled-up structure}\label{sec:plan}

In this section, we will define the scaling-up operation 
\( \cM \mapsto \cM^{*} \) producing \( \cM^{k+1} \) from  \( \cM^{k} \); however,
we postpone to Section~\ref{sec:params}
the definition of several parameters and probability bounds for \( \cM^{*} \).

 \subsection{The scale-up construction}\label{subsec:plan.constr}

Some of the following parameters will be given values only later, but they
are introduced by name here.

 \begin{definition}\label{def:f-g}
The positive parameters \( \bub,\g,\f \) will be different for each level of the
construction, and satisfy
 \begin{equation}\label{eq:bub-g-f}
   \begin{aligned}
         \bub / \g &= (\g/\f)^{1/2} \ll \slblb^{4},
\\    \f &\ll \bub^{*}.
   \end{aligned}
 \end{equation}
More precisely the \( \ll \) is understood here as
\( \lim_{\R\to\infty}\f/\bub^{*}=0 \).
 \end{definition}
Here is the approximate meaning of these parameters:
Walls closer than \( \f \) to each other, and
intervals larger than \( \g \) without holes raise alarm,
and a trap closer than \( \g \) makes a point unclean.
(The precise equality of the quotients above is not crucial for the proof, but
is convenient.) 

 \begin{definition}\label{def:new-slb}
Let \( \slb^{*}_{i} = \slb_{i} + \slopeincr\slblb^{-3}\bub/\g \)
for \( i=x,y \), where \( \slopeincr \)
is a constant to be defined later (in the proof of Lemma~\ref{lem:three}).
 \end{definition}

For the new value of \( \R \) we require
 \begin{equation}\label{eq:R-cond}
   \R^{*} \le 2 \R - \log_{\lg} \f.
 \end{equation}

 \begin{definition}[Light and heavy]
Barriers and walls of rank lower than \( \R^{*} \) are called \df{light}, the
other ones are called \df{heavy}.
 \end{definition}

Heavy walls of \( \cM \) will also be walls of \( \cM^{*} \) (with some exceptions given
below).
We will define walls only for either \( X \) or \( Y \), but it is understood
that they are also defined when the roles of \( X \) and \( Y \) are reversed.

The rest of the scale-up construction will be given in the following steps.

 \begin{defstep}[Cleanness]\label{defstep:clean}
For an interval \( I \), its right endpoint \( x \) will be called
clean in \( I \) for \( \cM^{*} \) if 
 \begin{enumerate}[--]
  \item It is clean in \( I \) for \( \cM \).
  \item The interval \( I \) contains no wall
of \( \cM \) whose right end is closer to \( x \) than \( \f/3 \).
 \end{enumerate}
We will say that a point is strongly clean in \( I \) for \( \cM^{*} \) if
it is strongly clean in \( I \) for \( \cM \) and \( I \) contains no barrier
of \( \cM \) whose right end is closer to it than \( \f/3 \).
Cleanness and strong cleanness of the left endpoint is defined similarly.

Let a point \( u \) be a starting point or endpoint of a rectangle \( Q \).
It will be called trap-clean in \( Q \) for \( \cM^{*} \) if
 \begin{enumerate}[--]
  \item It is trap-clean in \( Q \) for \( \cM \).
  \item Any trap contained in \( Q \) is at a distance \( \ge \g \) from \( u \).
 \end{enumerate}
 \end{defstep}

 \begin{defstep}[Uncorrelated traps]\label{defstep:uncorrel}
A rectangle \( Q \) is called an \df{uncorrelated compound trap} if
it contains two traps with disjoint projections, with a distance of their
starting points at most \( \f \), and if it is minimal among the rectangles
containing these traps.
 \end{defstep}
Clearly, the size of an uncorrelated trap is bounded by \( \bub+\f \).

 \begin{defstep}[Correlated trap]\label{defstep:correl}
  Let 
 \begin{equation}\label{eq:Lt}
 \L_{1}=29\slblb^{-1}\bub,\quad \L_{2}=9\slblb^{-1}\g.
 \end{equation}
(Choice motivated by the proof of Lemmas~\ref{lem:clean} and~\ref{lem:approx}.)
Let \( I \) be a closed interval with length \( \L_{i} \), \( i=1,2 \),
and \( b\in\bbZ_{+} \), with \( J=\clint{b}{b+5 \bub} \).
We say that event 
  \[
   \cL_{i}(X,Y,I,b)
  \]
  holds if \( I\times J \) contains at least four traps with disjoint \( x \) projections.
Let \( x(I),y(J) \) be given.
We will say that \( I \times J \) is a \df{horizontal correlated trap} of kind \( i \)
if \( \cL_{i}(X,Y,I,b) \) holds and
\begin{equation*}
  \Prob\paren{\cL_{i}(X,Y,I,b) \mid X(I) = x(I)} \le \tub^{2}.
\end{equation*}
Vertical correlated traps are defined analogously.
Figure 11 of~\cite{GacsWalks11} illustrates correlated traps.
 \end{defstep}

 \begin{remark}
In the present paper, traps of type 1 are used only in
Part~\ref{approx.u} of the proof of Lemma~\ref{lem:approx}.
 \end{remark}

 \begin{defstep}[Traps of the missing-hole kind]\label{defstep:missing-hole}
Let \( I \) be a closed interval of size \( \g \), let \( b \) be
a site with \( J = \clint{b}{b + 3\bub} \).
We say that event 
 \[
  \cL_{3}(X,Y,I,b)
 \]
holds if there is a \( b'>b+\bub \) such that
\( \rint{b+\bub}{b'} \) is the body of
a light horizontal potential wall \( W \), and no good vertical hole
(in the sense of Definition~\ref{def:holes}) 
\( \rint{a_{1}}{a_{2}} \) with \( \rint{a_{1}-\bub}{a_{2}+\bub} \sbsq I \) 
passes through \( W \).

Let \( x(I),y(J) \) be fixed.
We say that \( I \times J \) is a \df{horizontal trap of the missing-hole kind}  
if event \( \cL_{3}(X,Y,I,b) \) holds and 
 \begin{equation*}
 \Prob\bigparen{\cL_{3}(X,Y,I,b) \mid X(I) = x(I)} \le \tub^{2}.
 \end{equation*}
Figure 12 of~\cite{GacsWalks11} illustrates traps of the missing-hole kind.
 \end{defstep}
Note that the last probability is independent of the value of \( b \).

The value \( \L_{2} \) bounds the size of all new traps, and it is \( \ll\f \)
due to~\eqref{eq:bub-g-f}.

 \begin{defstep}[Emerging walls]\label{defstep:emerg}
We define some objects as barriers, and then designate some of the barriers (but
not all) as walls.

A vertical emerging barrier is, essentially, a horizontal interval over which
the conditional probability of a bad event \( \cL_{j} \) is not
small (thus preventing a new trap).
But in order to find enough barriers, the ends are allowed to be slightly extended.
Let \( x \) be a particular value of the sequence \( X \) over an 
interval \( I=\rint{u}{v} \).
For any \( u' \in \rint{u}{u+2\bub} \), \( v' \in \rint{v-2\bub}{v} \),
let us define the interval \( I'=\clint{u'}{v'} \).
We say that interval \( I \) 
is the body of a vertical \df{barrier} of the \df{emerging kind}, of type 
\( j\in\{1,2,3\} \) if the following inequality holds:
 \begin{equation*}
 \sup_{I'} \Prob\bigparen{\cL_{j}(x,Y,I',1) \mid X(I') = x(I')}  > \tub^{2}.  
 \end{equation*}
To be more explicit, for example interval \( I \) is an emerging barrier of
type \( 2 \) for the process \( X \)
if it has a closed subinterval \( I' \) of size \( \L_{2} \) within \( 2\bub \) of its
two ends, such that conditionally over the value of \( X(I') \), with
probability \( >\tub^{2} \), the rectangle
\(  I\times \clint{b}{b+5 \bub} \) contains four traps with disjoint \( x \) projections.
More simply, the value \( X(I') \) makes not too improbable (in terms of a randomly
chosen \( Y \)) for a sequence of closely placed traps to exist reaching horizontally across
\( I'\times \clint{b}{b+5 \bub} \).

Let 
 \[
  \L_{3}=\g.
 \]
Then emerging barriers of type \( j \) have length in 
\( \L_{j}+\clint{0}{4\bub} \).
Figure 13 of~\cite{GacsWalks11} illustrates emerging barriers.

We will designate some of the emerging barriers as walls.
We will say that \( I \) is a \df{pre-wall} of the emerging kind if also the
following properties hold:
 \begin{alphenum} 

  \item\label{i:emerg.hop}
Either \( I \) is an external hop of \( \cM \) or it is the union of a dominant
light wall and one
or two external hops of \( \cM \), of size \( \ge\bub \), surrounding it.

  \item\label{i:emerg.outside}
Each end of \( I \) is adjacent to either an external hop of size \( \ge\bub \)
or a wall of \( \cM \).
 \end{alphenum}
Figure 14 of~\cite{GacsWalks11} illustrates pre-walls.

\begin{sloppypar}
Now, for \( j=1,2,3 \), list all emerging pre-walls of type \( j \) in a sequence
\( \tup{B_{j1},B_{j2},\dots} \). 
First process the sequence \( \tup{B_{1 1},B_{1 2},\dots} \).
Designate \( B_{1 n} \) a wall if and only if it
is disjoint of all emerging pre-walls designated as walls earlier.
Then process the sequence \( \tup{B_{3 1},B_{3 2},\dots} \).
Designate \( B_{3 n} \) a wall if and only if it
is disjoint of all emerging pre-walls designated as walls earlier.
Finally process the sequence \( \tup{B_{2 1},B_{2 2},\dots} \) similarly.
\end{sloppypar}

To emerging barriers and walls, we assign rank 
 \begin{equation}\label{eq:emerg-rank}
  \hat\R > \R^{*}
 \end{equation}
to be determined later.

 \end{defstep}

 \begin{defstep}[Compound walls]\label{defstep:compound}
A \df{compound barrier} occurs in \( \cM^{*} \) for \( X \)
wherever barriers \( W_{1}, W_{2} \) occur (in this order) for \( X \)
at a distance 
\( d \le \f \), and \( W_{1} \) is light.
(The distance is measured
between the right end of \( W_{1} \) and the left end of \( W_{2} \).)
We will call this barrier a wall if \( W_{1},W_{2} \) are neighbor walls
(that is, they are walls separated by a hop).
We denote the new compound wall or barrier by
 \[
  W_{1} + W_{2}.
 \]
Its body is the smallest right-closed interval containing the bodies of
\( W_{j} \). 
For \( \r_{j} \) the rank of \( W_{j} \), we will say that the compound wall or
barrier in question has \df{type} 
 \[
   \ang{\r_{1},\r_{2},i},\text{ where } 
 i=
\begin{cases}
      d & \text{if } d\in\{0,1\},
\\  \flo{\log_{\lg}d} &\text{otherwise}.  
\end{cases}
 \]
Its rank is defined as
 \begin{equation}\label{eq:compound-rank}
  \r  = \r_{1} + \r_{2} - i.
 \end{equation}
Thus, a shorter distance gives higher rank.
This definition gives 
  \[
    \r_{1}+\r_{2}-\log_{\lg}\f\le \r \le \r_{1}+\r_{2}.
 \]
 Inequality~\eqref{eq:R-cond} will make sure that the rank of 
the compound walls is lower-bounded by \( \R^{*} \).

Now we repeat the whole compounding step, introducing compound walls and
barriers in which now \( W_{2} \) is required to be light.
The barrier \( W_{1} \) can be any barrier introduced until now, also a compound
barrier introduced in the first compounding step.
 \end{defstep}

The walls that will occur as a result of the compounding operation are 
of the type \( L +W \), \( W+L \), or \( (L+W)+L \), where \( L \) is a light wall
of \( \cM \) and \( W \) is any wall of \( \cM \) or an emerging wall of
\( \cM^{*} \).
Figure 15 of~\cite{GacsWalks11} illustrates the different kinds of compound barriers.
Thus, the maximum size of a compound wall is
 \[
   \bub + \f + (\L_{2}+4\bub) + \f + \bub < \bub^{*},
 \]
for sufficiently large \( \R_{1} \), where we used~\eqref{eq:bub-g-f}.

 \begin{defstep}[Finish]\label{defstep:finish}
The graph \( \cG \) does not change in the scale-up: \( \cG^{*}=\cG \).
Remove all traps of \( \cM \).

Remove all light walls and barriers.
If the removed light wall was dominant, remove also all other
walls of \( \cM \) (even if not light) contained in it.
 \end{defstep}

 \subsection{Combinatorial properties}\label{subsec:plan.prop}

The following lemmas are taken straight from~\cite{GacsWalks11}, and their
proofs are unchanged in every essential respect.

 \begin{lemma}\label{lem:distr.indep}
The new mazery \( \cM^{*} \) satisfies
Condition \ref{cond:distr}.\ref{i:distr.indep}.
 \end{lemma}

This lemma corresponds to Lemma~4.1 of~\cite{GacsWalks11}.

 \begin{lemma}\label{lem:cover}
The mazery \( \cM^{*} \) satisfies 
conditions~\ref{cond:distr}.\ref{i:distr.inner-clean}
and~\ref{cond:distr}.\ref{i:distr.cover}.
 \end{lemma}

This lemma corresponds to Lemma~4.2 of~\cite{GacsWalks11}.

 \begin{lemma}\label{lem:new-hop}
Suppose that interval \( I \) 
contains no walls of \( \cM^{*} \), and no wall of \( \cM \) closer to its ends than
\( \f/3 \) (these conditions are satisfied if it is a hop of \( \cM^{*} \)).
Then it either contains no walls of \( \cM \) or all walls of \( \cM \) in it are
covered by a sequence \( W_{1},\dots,W_{n} \) of dominant 
light neighbor walls of \( \cM \) separated from each other 
by external hops of \( \cM \) of size \( >\f \).

If \( I \) is a hop of \( \cM^{*} \) then either it is also a hop of \( \cM \) or the above
end intervals are hops of \( \cM \).
 \end{lemma}

This lemma corresponds to Lemma~4.3 of~\cite{GacsWalks11}.

 \begin{lemma}\label{lem:emerg-wall}
Let us be given intervals \( I' \subset I \), and also \( x(I) \),

with the following properties for some \( j\in\{1,2,3 \} \).
 \begin{alphenum}
  \item All walls of \( \cM \) in \( I \) are covered by a sequence \( W_{1},\dots,W_{n} \)
of dominant light neighbor walls of \( \cM \) such that the \( W_{i} \) are at a
distance \( >\f \) from each other and at a distance \( \ge\f/3 \) from the ends of
\( I \).
  \item  \( I' \) is an emerging barrier of type \( j \).
  \item\label{i:emerg-wall.away}
\( I' \) is at a distance \( \ge \L_{j}+7\bub \) from the ends of \( I \).
 \end{alphenum}
  Then \( I \) contains an emerging wall.
 \end{lemma}

This lemma corresponds to Lemma~4.4 of~\cite{GacsWalks11}.

 \begin{lemma}\label{lem:if-not-missing-hole}
Let the rectangle \( Q \) with \( X \) projection \( I \) contain no traps or vertical walls
of \( \cM^{*} \), 
and no vertical wall of \( \cM \) closer than \( \f/3 \) to its sides.
Let \( I' = \clint{a}{a + \g} \), \( J = \clint{b}{b + 3\bub} \)
with \( I' \times J \sbsq Q \) be such that \( I' \) is 
at a distance \( \ge\g+7\bub \) from the ends of \( I \).
Suppose that a light horizontal wall \( W \) starts at position \( b + \bub \).
Then \( \clint{a+\bub}{a+\g-\bub} \) 
contains a vertical hole passing through \( W \) that is good in the
sense of Definition~\ref{def:holes}.
The same holds if we interchange horizontal and vertical.
 \end{lemma}

This lemma corresponds to Lemma~4.5 of~\cite{GacsWalks11}.

 \begin{lemma}\label{lem:if-correl}
Let rectangle \( Q \) with \( X \) projection \( I \) contain no traps or vertical walls of
\( \cM^{*} \), and no vertical walls of \( \cM \) closer than \( \f/3 \) to its sides.
Let \( \L_{j} \), \( j=1,2 \) be as introduced in the definition of
correlated traps and emerging walls in Steps~\ref{defstep:correl}
and~\ref{defstep:emerg} of the scale-up construction.
Let \( I' = \clint{a}{a + \L_{j}} \),
\( J = \clint{b}{b + 5\bub} \) with \( I' \times J \sbsq Q \) be such that \( I' \) is 
at a distance \( \ge \L_{j}+7\bub \) from the ends of \( I \).
Then \( I' \) contains a subinterval \( I''  \) of size \( \L_{j}/4-2\bub \)
such that the rectangle \( I'' \times J \) contains no trap of \( \cM \).
The same holds if we interchange horizontal and vertical.
 \end{lemma}
This lemma corresponds to Lemma~4.6 of~\cite{GacsWalks11}.
Note that
\begin{align*}
   \L_{2}/4-2\bub>2.2\slblb^{-1}\g.
 \end{align*}

 \begin{lemma}\label{lem:clean}
The new mazery \( \cM^{*} \) defined by the above construction satisfies
Conditions~\ref{cond:distr}.\ref{i:distr.clean.1}
and~\ref{cond:distr}.\ref{i:distr.clean.2}.
 \end{lemma}
This lemma corresponds to Lemma~4.7 of~\cite{GacsWalks11}.

\section{The scale-up functions}\label{sec:params}

Mazery \( \cM^{1} \) is defined in Example~\ref{xmp:base}.
The following definition introduces some of the parameters needed for scale-up.
The choices will be justified by the lemmas of Section~\ref{sec:bounds}.

 \begin{definition}\label{def:ranks}
At scale-up by one level, to
obtain the new rank lower bound, we multiply \( \R \) by a constant:
 \begin{equation}\label{eq:txp}
           \R = \R_{k}  = \R_{1} \txp^{k-1},
\quad       \R_{k+1} = \R^{*} = \R \txp,
\quad                1 < \txp < 2.
 \end{equation}
The rank of emerging walls, introduced in~\eqref{eq:emerg-rank}, is defined
using a new parameter \( \txp' \):
 \[
   \hat\R = \txp'\R.
 \]
 \end{definition}

We require
 \begin{equation}\label{eq:txp'}
 \txp < \txp' < \txp^{2}.
 \end{equation}

We need some bounds on the possible rank values.

 \begin{definition} Let \( \txpub = 2\txp / (\txp - 1) \).
 \end{definition}

  \begin{lemma}[Rank upper bound]\label{lem:rank-bds}
In a mazery, all ranks are upper-bounded by \( \txpub \R \).
  \end{lemma}
This lemma and its corollary correspond to Lemma~6.1 and Corollary~6.2 
of~\cite{GacsWalks11}.
 \begin{corollary}\label{c.rank-lifetime}
Every rank exists in \( \cM^{k} \) for at most
\( \cei{\log_{\txp}\frac{2 \txp}{\txp - 1}} \) values of \( k \).
 \end{corollary}

It is convenient to express 
several other parameters of \( \cM \) and the scale-up in terms of a single
one, \( \T \):

 \begin{definition}[Exponential relations]\label{def:exponential}
 Let \( \T = \lg^{\R} \),
 \begin{align*}
      \bub     &= \T^{\bubxp},
\quad     \g        = \T^{\gxp},
\quad     \f        = \T^{\fxp},
\quad     \tub      = \T^{-\tubxp}.
 \end{align*}
We require
 \begin{equation}\label{eq:bubxp-etc}
           0 < \bubxp < \gxp 
           < \fxp < 1.
 \end{equation}
 \end{definition}
Note that the requirement~\eqref{eq:R-cond} is satisfied as long as
\begin{align}\label{eq:txp-ub}
\txp \le 2 - \fxp.
 \end{align}
Our definitions give \( \bub^{*} = \bub^{\txp} \).
Let us see what is needed for this to indeed upperbound the size of any new
walls in \( \cM^{*} \).
Emerging walls can have size as large as \( \L_{2}+ 4\bub \),
and at the time of their creation, they are the largest existing ones.
We get the largest new walls when the compound operation combines these
with light walls on both sides, leaving the largest gap possible,
so the largest new wall size is
 \[
  \L_{2} + 2\f + 6\bub < 3\f,
 \]
where we used \( \bub\ll\g\ll\f \) 
from~\eqref{eq:bub-g-f}, and that \( \R_{1} \) is large enough.
In the latter case, we always get \( 3\f\le\bub^{*} \) if
 \begin{equation}\label{eq:fxp-ub}
  \fxp < \txp \bubxp.
 \end{equation}
As a reformulation of~\eqref{eq:bub-g-f}, we require
 \begin{equation}\label{eq:bub-g-f-mod}
 2(\gxp-\bubxp)=\fxp-\gxp.
 \end{equation}
We also need
 \begin{align} \label{eq:trap-xp}
   2\gxp-\txp\bubxp+1 &<\tubxp,
\\\label{eq:correl-trap-xp}
                               4(\gxp + \bubxp) &< \tubxp(4 - \txp),
\\\label{eq:emerg-xp.2}         4\gxp + 6\bubxp + \txp'  &< 2\tubxp,
\\\label{eq:emerg-xp.3}         \txp(\bubxp+1)   &< \txp'.
 \end{align}
(Lemma~\ref{lem:ncln-ub} uses~\eqref{eq:trap-xp},
Lemma~\ref{lem:trap-scale-up} uses~\eqref{eq:correl-trap-xp},
Lemma~\ref{lem:emerg-contrib} uses~\eqref{eq:emerg-xp.2}, and
Lemma~\ref{lem:emerg-hole-lb} uses~\eqref{eq:emerg-xp.3}.)

Using the exponent \( \hxp \) introduced in~\eqref{eq:hxp}, we require
 \begin{align}
\label{eq:hxp-ub.1}                  \txp\hxp     &< \gxp - \bubxp,
\\\label{eq:hxp-ub.2}              \txpub\hxp &< 1 - \txp\bubxp,
\\\label{eq:hxp-ub.3}              \txpub\hxp &< \tubxp - 2 \txp\bubxp.
 \end{align}
(Lemmas~\ref{lem:trap-scale-up} and~\ref{lem:emerg-contrib} 
use~\eqref{eq:hxp-ub.1},
Lemmas~\ref{lem:all-compound-hole-lb} and~\ref{lem:heavy-hole-lb} 
use~\eqref{eq:hxp-ub.2}, and Lemma~\ref{lem:all-compound-hole-lb} 
uses~\eqref{eq:hxp-ub.3}.)

The condition these inequalities impose on \( \hxp \) is just
to be sufficiently small (and, of course, that the bounds involved are
positive).
On \( \tubxp \) the condition is just to be sufficiently large.

 \begin{lemma}\label{lem:exponents-choice}
 The exponents \( \bubxp, \gxp, 
\fxp, \txp, \txp', \hxp \) can be chosen to
satisfy the inequalities~\eqref{eq:txp}, \eqref{eq:txp'}, 
\eqref{eq:bubxp-etc}-\eqref{eq:hxp-ub.3}.
 \end{lemma}
 \begin{proof}
It can be checked that the choices
\( \bubxp=0.15 \), \( \gxp=0.18 \), 
\( \fxp=0.24 \), \( \txp=1.75 \), \(\txp'=2.5 \),
\( \tubxp=4.5 \), \( \txpub=4.66\dots \) satisfy all the inequalities in question.
 \end{proof}

 \begin{definition}
Let us fix now the exponents \( \bubxp, \fxp, \gxp, 
\txp, \txp', \hxp \)  as chosen in the lemma.
In order to satisfy all our requirements also for small \( k \),
we will fix \( \aux_{2} \) sufficiently small,
then \( \aux_{3} \) sufficiently
large, and finally \( \R_{1} \) sufficiently large.
 \end{definition}

We need to specify some additional parameters.

 \begin{definition}\label{def:new-ncln} Let
\( \ncln_{i}^{*} = \ncln_{i} + \bub^{*}\T^{-1} \) for \( i=\w,\t\).
 \end{definition}

In estimates that follow, in order to avoid cumbersome calculations, we will
liberally use the notation \( \ll \), \( \gg \), \( o() \), \( O() \).
The meaning is always in terms of \( \R_{1}\to\infty \).

\section{Probability bounds}\label{sec:bounds}

In this section, we derive the bounds on probabilities in \( \cM^{k} \), sometimes
relying on the corresponding bounds for \( \cM^{i} \), \( i<k \).

\subsection{New traps}

 \begin{lemma}[Uncorrelated Traps]\label{lem:uncorrel-trap}
Given a string \( x = \tup{x(0), x(1),\dots} \), a point \( \pair{a_{1}}{b_{1}} \),
let \( \cF \) be the event that an uncorrelated compound
trap of \( \cM^{*} \) starts at \( \pair{a_{1}}{b_{1}} \).
Then
 \begin{equation*}
 \Prob\bigparen{\cF \mid X = x} \le 2 \f^{2} \tub^{2}.
 \end{equation*}
 \end{lemma}

This lemma corresponds to Lemma~5.4 of~\cite{GacsWalks11}.

 \begin{lemma}[Correlated Traps]\label{lem:correl}
Let a site \( \pair{a}{b} \) be given.
For \( j=1,2 \), let \( \cF_{j} \) be the event that a horizontal correlated trap of type \( j \)
starts at \( \pair{a}{b} \).
 \begin{alphenum}

  \item\label{i:correl-trap-vert} 
Let us fix a string \( x = \tup{x(0), x(1), \dots} \).
We have
 \begin{equation*}
 \Prob\bigparen{\cF_{j} \mid X=x} \le w^{2}.
 \end{equation*}

  \item\label{i:correl-trap-horiz} 
Let us fix a string \( y = \tup{y(0), y(1), \dots} \).
We have
 \begin{equation*}
 \Prob\bigparen{\cF_{j} \mid Y = y} \le (5\bub\L_{j}\tub)^{4}.
 \end{equation*}
 \end{alphenum}
 \end{lemma}

This lemma corresponds to Lemma~5.5 of~\cite{GacsWalks11}.

Before considering missing-hole traps,
recall the definitions needed for the hole lower bound
condition, Condition \ref{cond:distr}.\ref{i:distr.hole-lb}, in particular the
definition of the numbers \( a,u,v,w,b,c \), and event \( E \).

Since we will hold the sequence \( y \) of values of the sequence \( Y \) of random
variables fixed in this subsection, we take the liberty
and omit the condition \( Y = y \) from the probabilities: it is always assumed
to be there.

Recall the definitions of events \( F \) and \( E \) in 
Condition \ref{cond:distr}.\ref{i:distr.hole-lb}.
For integers \( a \) and \( u\le v \) and a horizontal wall \( \rint{v}{w} \)
we defined \( b,c \) by appropriate formulas, and
for a \( d \in \clint{b}{c} \) the event \( F(u, v;\, a, d) \)
(a function of \( X \)) saying that
\( \Rect^{\rightarrow}(\pair{a}{u}, \pair{d}{v}) \)
contains no traps or vertical barriers, and is inner H-clean.
We elaborate now on the definition of event \( E(u,v,w;\,a) \) as follows.
For \( t> d \) let \( \tilde E(u,v,w;\, a,d,t) \) be the event that 
\( \rint{d}{t} \) is a hole fitting wall \( \rint{v}{w} \), and event \( F(u, v;\, a, d) \) holds.
Then event \( E(u, v, w;\, a) \) holds if there are \( d,t \) such that
event \( \tilde E(u,v,w;\, a,d,t) \) holds.
Let  \( \hat E(u, v, w;\, a) \) hold if there are \( d,t \) such that
event \( \tilde E(u,v,w;\, a,d,t) \) holds and the point \( \pair{t}{w} \) is 
upper right rightward H-clean
(that is the hole \( \rint{d}{t} \) is good as seen from \( \pair{a}{u} \), in the
sense of Definition~\ref{def:holes}).

 \begin{lemma}\label{lem:hole-lb-clean}
We have
 \begin{equation*}
   \Prob(\hat E) \ge (1-2\ncln_{\w}) \Prob(E) \ge 0.9\Prob(E).
 \end{equation*}
 \end{lemma}

This lemma corresponds to Lemma~5.1 of~\cite{GacsWalks11}.

 \begin{lemma}\label{lem:hole-lb-clean-2}
Let \( v < w \), and let us fix the value \( y \) of the sequence of random variables
\( Y \) in such a way that
there is a horizontal wall \( B \) of rank \( \r \), with body \( \rint{v}{w} \).
For an arbitrary integer \( b \), 
let \( G=G(v,w;b) \) be the event that a good hole through \( B \) starts at
position \( b \) (this event still depends on 
the sequence \( X=\tup{X(1),X(2),\dots} \) of random variables).
Then
 \[
   \Prob(G) \ge (1-\ncln_{\w}-\ncln_{\t})(1-2\ncln_{\w})\h(\r)\ge 0.3\h(\r).
 \]
 \end{lemma}

This lemma corresponds to Lemma~5.2 of~\cite{GacsWalks11}.

Recall the definition of traps of the missing-hole kind in
Step~\ref{defstep:missing-hole} of the scale-up algorithm in Section~\ref{sec:plan}.

 \begin{lemma}[Missing-hole traps]\label{lem:missing-hole}
For \( a,b \in \bbZ_{+} \), let \( \cF \) be the event 
that a horizontal trap of the missing-hole kind starts at \( \pair{a}{b} \).
 \begin{alphenum}
  \item\label{i:missing-hole-trap-vert}
Let us fix a string \( x = \tup{x(0), x(1), \dots} \).
We have
 \begin{equation*}
 \Prob\bigparen{\cF \mid X=x} \le w^{2}.
 \end{equation*}

  \item\label{i:missing-hole-trap-horiz}
Let us fix a string \( y = \tup{y(0), y(1), \dots} \).
Let \( n = \Flo{\frac{\g}{(\slblb^{-1}+2)\bub}} \).
We have 
 \begin{equation*}
 \Prob\bigparen{\cF \mid Y=y} \le e^{- 0.3 n \h(\R^{*})}.
 \end{equation*}
 \end{alphenum}
 \end{lemma}
This lemma corresponds to Lemma~5.6 of~\cite{GacsWalks11}.
There, we had \( (1-\ncln)^{2} \) in
place of \( 0.3 \) which stands here for \( (1-\ncln_{\w}-\ncln_{\t})(1-2\ncln_{\w}) \),
and \( n=\flo{\g/3\bub} \).
The latter change is needed here
since we use \( \slblb^{-1}\bub \) instead of \( \bub \) to upperbound the width of holes.
The proof is otherwise identical.


 \begin{lemma}\label{lem:trap-scale-up}
For any value of the constant \( \aux_{3} \), if \( \R_{1} \) is sufficiently large
then the following holds: 
if \( \cM=\cM^{k} \) is a mazery then \( \cM^{*} \) satisfies the trap upper 
bound~\ref{cond:distr}.\ref{i:distr.trap-ub}.
 \end{lemma}
This lemma corresponds to Lemma~7.1 of~\cite{GacsWalks11}.

\subsection{Upper bounds on walls}

Recall the definition of \( \p(\r) \) in~\eqref{eq:wall-prob}, used to upperbound
the probability of walls.
Recall the definition of emerging walls in
Step~\ref{defstep:emerg} of the scale-up algorithm in Section~\ref{sec:plan}.

 \begin{lemma}\label{lem:emerg}
For any point \( u \), let \( \cF(t) \) be the event that a barrier \( \rint{u}{v} \) 
of \( X \) of the emerging kind, of length \( t \), starts at \( u \).
Denoting \( n = \Flo{\frac{\g}{(\slblb^{-1}+2)\bub}} \) we have:
 \begin{equation*}
 \sum_{t}\Prob\paren{\cF(t)} \le 4\bub^{2} \tub^{2}
 \bigparen{2\cdot (5\bub\L_{2})^{4} + \tub^{-4} e^{-0.3 n \h(\R^{*})}}.
 \end{equation*}
 \end{lemma}
This lemma corresponds to Lemma~5.7 of~\cite{GacsWalks11}.
There we had \( (1-\ncln)^{2} \) in place
of \( 0.3 \), and \( n=\flo{\g/3\bub} \).
There was also a factor of \( m \) due to Markov conditioning
(with a meaning different from the present \( \m \)) that is not needed here.
The proof is otherwise identical.

 \begin{lemma}\label{lem:emerg-contrib}
For every possible value of 
\( \aux_{2},\aux_{3} \), if  \( \R_{1} \) is sufficiently large then the following holds.
Assume that \( \cM = \cM^{k} \) is a mazery.
Fixing any point \( a \), the
sum of the probabilities over \( l \) that a barrier of the emerging kind of size
\( l \) starts at \( a \) is at most \( \p(\hat\R) / 2 = \p(\txp'\R) / 2 \).
 \end{lemma}

This lemma corresponds to Lemma~7.2 of~\cite{GacsWalks11}.


Let us use the definition of compound walls given in
Step~\ref{defstep:compound} of the scale-up algorithm of
Section~\ref{sec:plan}.

 \begin{lemma}\label{lem:compound-wall-ub} 
Consider ranks \( \r_{1},\r_{2} \) at any stage of the scale-up construction.
Assume that Condition \ref{cond:distr}.\ref{i:distr.wall-ub} already holds for
rank values \( \r_{1},\r_{2} \).
For a given point \( x_{1} \) the sum, over all \( l \), of the probabilities for the
occurrence of a compound barrier of type \( \ang{\r_{1},\r_{2},i} \)
and width \( l \) at \( x_{1} \) is bounded above by
 \begin{equation*}
   \lg^{i}\p(\r_{1})\p(\r_{2}).
 \end{equation*}
 \end{lemma}
This lemma corresponds to Lemma~5.8 of~\cite{GacsWalks11}.

 \begin{lemma}\label{lem:compound-contrib}
For a given value of \( \aux_{2} \), if 
we choose the constant 
\( \R_{1} \) sufficiently large 
then the following holds.
Assume that \( \cM=\cM^{k} \) is a mazery.
After one operation of forming compound barriers, fixing any point \( a \),
for any rank \( \r \), the sum, over all widths \( l \), of the probability that
a compound barrier of rank \( \r \) and width \( l \) starts at \( a \) is at most 
\( \p(\r)\R^{-\aux_{1}/2} \).
 \end{lemma}
This lemma corresponds to Lemma~7.3 of~\cite{GacsWalks11}.

 \begin{lemma}\label{lem:all-wall-ub}
For every choice of 
\( \aux_{2},\aux_{3} \) if we choose \( \R_{1} \)
sufficiently large then the following holds.
Suppose that each structure \( \cM^{i} \) for \( i \le k \) is a mazery.
Then Condition \ref{cond:distr}.\ref{i:distr.wall-ub} holds for \( \cM^{k+1} \).
\end{lemma}

This lemma corresponds to Lemma~7.4 of~\cite{GacsWalks11}.

 \begin{lemma}\label{lem:pub}
For small enough \( \aux_{2} \), the probability of a barrier
of \( \cM \) starting at a given point \( b \) is bounded by \( \T^{-1} \).
 \end{lemma}

This lemma corresponds to Lemma~7.5 of~\cite{GacsWalks11}
(where the intermediate notation \( \overline p \) was
also used for the upper bound).

\subsection{Lower bounds on holes}\label{subsec:bounds}

Before proving the hole lower bound condition for \( \cM^{*} \),
let us do some preparation.

 \begin{definition}\label{def:E-star}
Recall the definition of event \( E \) in
Condition \ref{cond:distr}.\ref{i:distr.hole-lb}, and that it refers to a
horizontal wall with body \( \rint{v}{w} \) seen from a point \( \pair{a}{u} \).
Take the situation described above, possibly without the bound on \( v-u \).

Let event \( F^{*}(u,v;\, a, d) \) 
(a function of the sequence \( X \)) be defined just as the event 
\( F(u,v;\, a,d) \) in Condition \ref{cond:distr}.\ref{i:distr.hole-lb}, 
except that the part requiring inner H-cleanness and freeness from traps and
barriers of the rectangle \( Q(d) \) must now be understood in the sense of both
\( \cM \) and \( \cM^{*} \).
Let
 \begin{equation*}
  E^{*} = E^{*}(u, v, w;\, a)
 \end{equation*}
be the event that there is a \( d \in \clint{b}{c} \) where a vertical hole fitting
wall \( B=\rint{v}{w} \) starts, and event \( F^{*}(u,v;\,a,d) \) holds.
 \end{definition}

Note that in what follows we will use the facts several times that
 \begin{align*}
   \tub_{k+1}<\tub_{k},\; \T_{k+1}>\T_{k},
 \end{align*}
in other words that
the bound \( \tub_{k} \)
on the conditional probability of having a trap at some point in \( \cM^{k} \)
serves also as a bound on the conditional
probability of having one in \( \cM^{k+1} \), and similarly with the bound \( \T_{k}^{-1} \)
for walls.

 \begin{lemma}\label{lem:hole-lb-2}
Suppose that the requirement \( v-u \le \slblb^{-2}\bub \) 
in the definition of the event \( E^{*} \) is replaced with
 \( v-u\le \slblb^{-2}\bub^{*} \),
while the rest of the requirements are the same.
Then we have
 \begin{equation*}
   \Prob\paren{E^{*}\mid Y=y} \ge 0.25 \land (v-u+1)^{\hxp} \h(\r) - U,
 \end{equation*}
where \( U = \T^{-\txpub\hxp-\eps} \) for some constant \( \eps>0 \).
If \( v-u>\slblb^{-2}\bub \) then we also have the somewhat stronger
inequality 
 \begin{equation*}
   \Prob\paren{E^{*}\mid Y=y} \ge 0.25 \land 2(v-u+1)^{\hxp} \h(\r) - U.
 \end{equation*}
The same statement holds if we replace horizontal with vertical.
 \end{lemma}
 \begin{Proof}
This lemma corresponds to Lemma~5.3 of~\cite{GacsWalks11}
(incorporating the estimate of
the expression called \( U \) there), but there are some
parameter refinements due to the refined form of 
Condition \ref{cond:distr}.\ref{i:distr.hole-lb}.
For ease of reading, we will omit the condition \( Y=y \) from the
probabilities. 
We will make the proof such that it works also if we interchange horizontal and
vertical, even though \( \slb_{x}\ne\slb_{y} \).

Consider first the simpler case, showing
that \( v-u\le \slblb^{-2}\bub \) implies
\( \Prob(E^{*}) \ge (v-u+1)^{\hxp} \h(\r) \).
Condition~\ref{cond:distr}.\ref{i:distr.hole-lb} implies this already for
\( \Prob(E) \), so it is sufficient to show \( E\sbsq E^{*} \) in this case.
As remarked after its definition, the event \( E^{*} \) differs from \( E \) only in
requiring that rectangle \( Q \) contain no traps or vertical barriers
of \( \cM^{*} \), not only of \( \cM \), and that 
points \( \pair{a}{u} \) and \( \pair{d}{v} \) are H-clean in \( Q \) 
for \( \cM^{*} \) also, not only for \( \cM \).

A trap of \( \cM^{*} \) in \( Q \) cannot be an uncorrelated or
correlated trap, since its components traps, being traps of \( \cM \), are
already excluded.
It cannot be a trap of the missing-hole kind either, since that trap, 
of length \( \g \) on one side, is too big
for \( Q \) when \( v-u\le \slblb^{-2}\bub \), and \( c-a \) is
also of the same order.
The same argument applies to vertical barriers of \( \cM^{*} \).
The components of the compound barriers that belong to \( \cM \)
are excluded, and the emerging barriers are too big, of
the size of correlated or missing-hole traps.

These considerations take care also of the issue of H-cleanness
for \( \cM^{*} \), since the latter also boils down to the absence of traps and
barriers.

Take now the case \( v-u>\slblb^{-2}\bub \).
Let
 \begin{align*}
   u'    &= v - \flo{\bub},
\quad \bub' = \flo{4 \slblb^{-1}\bub}, 
\\        n &= \cei{(c-b)/\bub'},
\\  a_{i} &= b + i\bub',
\quad    E'_{i} = E(u', v, w;\, a_{i})
\quad \txt{ for } i = 0, \dots, n-1,
\\     E'  &= \bigcup_{i} E'_{i}.
 \end{align*}
From~\eqref{eq:slb-lb} and~\eqref{eq:slb-ub} follows 
\( \slb_{x}^{-1}-\slb_{y}\ge\slblb/\slb_{x} \).
Recall
 \begin{align*}
                      b &= a + \cei{\slb_{y}(v-u)},
\\                  c &= b\lor (a+\flo{\slb_{x}^{-1}(v-u)}).
 \end{align*}
Hence
 \begin{align}\nonumber
  c-b &\ge (\slb_{x}^{-1}-\slb_{y})(v-u)-2 \ge (v-u)\slblb/\slb_{x}-2,
\\\label{eq:n-lb}   
   n   & \ge ((v-u)\slblb/\slb_{x} - 2)/\bub' \ge (v-u+1)\slblb/5\bub 
   \ge \slblb^{-1}/5,
 \end{align}
where the factor \( 1/5 \) instead of \( 1/4 \) allows omitting the 
\( -2 \) and adding the \( +1 \), and ignoring the integer part in \( \bub' \).
Let \( C \) be the event that point \( \pair{a}{u} \) is upper right rightward H-clean in
\( \cM \).
Then by Conditions~\ref{cond:distr}.\ref{i:distr.ncln}
 \begin{align}\label{eq:P(neg C)}
 \Prob(\neg C) \le 2\ncln_{\w}\le 0.1.
 \end{align}
Let \( D \) be the event that the rectangle \( \rint{a}{c} \times \clint{u}{v} \)
contains no trap or vertical barrier of \( \cM \) or \( \cM^{*} \).
(Then \( C\cap D \) implies that \( \pair{a}{u} \) is also upper right rightward H-clean
in the rectangle \( \rint{a}{c} \times \clint{u}{v} \)
in \( \cM^{*} \).)
By Lemmas~\ref{lem:trap-scale-up}, \ref{lem:pub}:
 \begin{align*}
  \Prob(\neg D) &\le 2(c-a)\T^{-1} + 2(c-a)(v-u+1)\tub.
 \end{align*}
Now
 \begin{align*}
 2(c-a) &\le 2\slblb^{-2}\slb_{x}^{-1}\bub^{*} \le 2\slblb^{-3}\bub^{*},
\\ v-u+1&\le \slblb^{-2}\bub^{*}+1\le 2\slblb^{-2}\bub^{*},    
 \end{align*}
hence
\begin{align*}
\Prob(\neg D)  
&\le 2\slblb^{-3}\bub^{*}\T^{-1}
+ 6\slblb^{-5}(\bub^{*})^{2}\tub
\\ &=  2\slblb^{-3}\T^{\txp\bubxp-1}
 + 6\slblb^{-5}\T^{-\tubxp+2\txp\bubxp}
\\ &< (2\slblb^{-3}+6\slblb^{-5})\T^{-\txpub\hxp-2\eps},
 \end{align*}
where \( \eps>0 \) is a constant and we used~(\ref{eq:hxp-ub.2}-\ref{eq:hxp-ub.3}).
Now the statement follows since \( \T^{-\eps}=\lg^{-\R\eps} \) decreases to 0 faster as a function
of \( \m \) than the expression in parentheses in front.

 \begin{step+}{step:hole-lb-2.CD}
Let us show \( C \cap D \cap E' \sbsq E^{*}(u, v, w;\, a) \).
 \end{step+}
 \begin{prooof}
Indeed, suppose that \( C \cap D \cap E'_{i} \) holds with some hole starting
at \( d \). 
Then there is a 
rectangle \( Q'_{i} = \Rect^{\rightarrow}(\pair{a_{i}}{u'}, \pair{d}{v}) \)
containing no traps or vertical barriers of \( \cM \),
such that \( \pair{d}{v} \) is H-clean in \( Q'_{i} \).
It follows from \( D \) that the rectangle 
 \[
  Q^{*}_{i} = \Rect^{\rightarrow}(\pair{a}{u}, \pair{d}{v}) \spsq Q'_{i}
 \]
contains no traps or vertical barriers of \( \cM \) or \( \cM^{*} \).
Since event \( C \) occurs, the point \( \pair{a}{u} \) is H-clean for \( \cM \) in
\( Q^{*}_{i} \).
The event \( E'_{i} \) and the inequalities \( d-a_{i},v-u'\ge\bub \) 
imply that \( \pair{d}{v} \) is H-clean in \( Q^{*}_{i} \),
and a hole passing through the potential wall starts at \( d \) in \( X \).
The event \( D \) implies that
there is no trap or vertical barrier of \( \cM \) in \( Q^{*}_{i} \).
Hence \( Q^{*}_{i} \) is also inner H-clean in \( \cM^{*} \), and so \( E^{*} \) holds.
 \end{prooof}

We have \( \Prob(E^{*}) \ge \Prob(C)\Prob(E' \mid C) - \Prob(\neg D) \).

\begin{step+}{step:hole-lb-2.E'}
The events \( E'_{i} \) are independent of each other and of the event \( C \).
\end{step+}
\begin{pproof}
By assumption, \( v-u>\slblb^{-2}\bub \), so
\( b-a\ge\slb_{y}(v-u)\ge\slb_{y}\slblb^{-2}\bub\ge\bub \), hence the event \( C \)
depends only on the part of the process \( X \) before point \( b \).
This shows that the events \( E'_{i} \) are independent of \( C \).
The hole starts within 
\( \slb_{x}^{-1}(v-u') \le \slb_{x}^{-1}\bub \)  after \( a_{i} \).  
The width of the hole through the wall \( B \) is at most \( \slblb^{-1}\bub \).
After the hole,
the property that the wall be upper right rightward H-clean depends on at
most \( \bub \) more values of \( X \) on the right.
So the event \( E_{i} \) depends at most on \( (2\slblb^{-1}+1)\bub<\bub' \)
values of the sequence \( X \) on the right of \( a_{i} \).
\end{pproof} 

 \begin{step+}{step:hole-lb-2.union-prob}
It remains to estimate \( \Prob(E' \mid C)=\Prob(E') \).
 \end{step+}
 \begin{prooof}
The following inequality can be checked by direct calculation.
Let \( \ag = 1 - 1/e = 0.632\dots \), then for \( x > 0 \) we have
 \begin{equation}\label{eq:exp-lb}
   1 - e^{-x} \ge \ag \land \ag x.
 \end{equation}
Condition \ref{cond:distr}.\ref{i:distr.hole-lb} is applicable to
\( E'_{i} \), so we have
 \begin{align*}
\Prob\paren{E'_{i}} \ge \bub^{\hxp}\h(\r) =\vcentcolon\s,   
 \end{align*}
hence \( \Prob\paren{\neg E'_{i}} \le 1 - \s \le e^{-\s} \).
Due to the independence of the sequence \( X \), this implies
 \begin{equation}\label{eq:union-E-R}
  \Prob(E') = 1 - \Prob\bigparen{\textstyle \bigcap_{i}\neg E'_{i}}
                 \ge 1 - e^{- n \s} \ge \ag \land \ag n \s,
 \end{equation}
where we used~\eqref{eq:exp-lb}.
Using~\eqref{eq:n-lb} twice (for lowerbounding \( n \) and \( n\bub \)):
 \begin{align*}
 n \bub^{\hxp} &= n^{1-\hxp}(n\bub)^{\hxp}
\\ &\ge (\slblb^{-1}/5)^{1-\hxp}5^{-\hxp}\slblb^{\hxp}(v-u+1)^{\hxp}
         = 5^{-1}\slblb^{2\hxp-1}(v-u+1)^{\hxp}.
 \end{align*}
Substituting into~\eqref{eq:union-E-R}:
 \begin{align*}
   \Prob(E') &\ge
   \ag\land\ag\cdot 5^{-1}\slblb^{2\hxp-1}(v-u+1)^{\hxp}\h(\r),
\\ \Prob(C)\Prob(E') &
\ge 0.9 \cdot \bigparen{\ag \land \ag\cdot 5^{-1}\slblb^{2\hxp-1}(v-u+1)^{\hxp}\h(r)}
\\   &\ge 0.5\land 2 (v-u+1)^{\hxp}\h(\r)
 \end{align*}
where we used~\eqref{eq:P(neg C)}.
 \end{prooof}
 \end{Proof}

The lower bound on the probability of holes through an emerging wall
is slightly more complex than
the corresponding lemma in~\cite{GacsWalks11}.
Recall \( F^{*} \) from Definition~\ref{def:E-star}.

\begin{lemma}\label{lem:hop-lb}
Using the notation of Condition \ref{cond:distr}.\ref{i:distr.hole-lb} for
\( \cM^{*} \), \( a,u,v,b \), assume that \( Y=y \) is fixed and \( v>u \).
Then \( \Prob\paren{F^{*}(u, v;\, a, b)}\ge 0.25 \).
 \end{lemma}

This lemma corresponds to Lemma~7.8 of~\cite{GacsWalks11}.

 \begin{proof}
Consider the case of a horizontal wall, the argument also works for the case of a
vertical wall.
The probability that it is not inner H-clean is at most
\( \ncln_{\t}+3\ncln_{\w} \) (adding up the probability bounds for the inner
horizontal non-cleanness and the inner trap non-cleanness of the two endpoints).
The probability of finding a vertical barrier or trap (of \( \cM \) or \( \cM^{*} \))
is bounded by \( U \) as in Lemma~\ref{lem:hole-lb-2}, so the total bound is at most 
\begin{align*}
   \ncln_{\t}+3\ncln_{\w}+U.
 \end{align*}
Here, \( U \) can be made less than \( 0.05 \) if \( \R_{1} \) is sufficiently large, so the
total is at most \( 0.75 \).
 \end{proof}

\begin{lemma}\label{lem:emerg-hole-lb}
For emerging walls, the fitting holes
satisfy Condition \ref{cond:distr}.\ref{i:distr.hole-lb}
if \( \R_{1} \) is sufficiently large.
\end{lemma}
This lemma corresponds to Lemma~7.9 of~\cite{GacsWalks11},
with Figure 22 there illustrating the proof.

Consider now a hole through a compound wall.
In the lemma below, we use \( w_{1},w_{2} \): please note that these are integer
coordinates,
and have nothing to do with the trap probability upper bound \( \tub \): we will
never have these two uses of \( w \) in a place where they can be confused.

\begin{lemma}\label{lem:compound-hole-lb}
Let \( u \le v_{1} < w_{2} \), and \( a \) be given with
\( v_{1}-u \le \slblb^{-2}\bub^{*} \).
Assume that \( Y=y \) is fixed in such a way that \( W \) is a compound
horizontal wall with body \( \rint{v_{1}}{w_{2}} \), and type \( \ang{\r_{1}, \r_{2}, i} \),
with rank \( \r \) as given in~\eqref{eq:compound-rank}.
Assume also that the component walls \( W_{1}, W_{2} \) already satisfy the hole
lower bound, Condition \ref{cond:distr}.\ref{i:distr.hole-lb}.
Let 
 \[
  E_{2} = E_{2}(u, v_{1}, w_{2};\, a) = E^{*}(u, v_{1}, w_{2};\, a)
 \]
where \( E^{*} \) was introduced in Definition~\ref{def:E-star}.
Assume
 \begin{equation}\label{eq:s-ub}
   (\slblb^{-2}\bub^{*}+1)^{\hxp} \h(\r_{j}) \le 0.25,\txt{ for } j=1,2.
 \end{equation}
Then
 \begin{equation*}
  \Prob\bigparen{E_{2}\mid Y=y} \ge 
   (v_{1}-u+1)^{\hxp}\lg^{i\hxp} \h(\r_{1}) \h(\r_{2})(1 - V)
 \end{equation*}
with \( V = 2U/ \h(\r_{1} \lor \r_{2}) \), where \( U \) comes from
Lemma~\ref{lem:hole-lb-2}.

The statement also holds if we exchange horizontal and vertical.
 \end{lemma}
The lemma corresponds to Lemma~5.9 of~\cite{GacsWalks11},
with Figure 21 there illustrating the proof.
Some parts of the proof 
are simpler, due to using \( v_{1}-u+1 \) in place of \( c-b \).
 \begin{Proof}
Let \( \D \) be the distance between the component walls
\( W_{1}, W_{2} \) of the wall \( W \), where the body of \( W_{j} \) is 
\( \rint{v_{j}}{w_{j}} \).
Consider first passing through \( W_{1} \).
For each integer \( x \in \clint{b}{c + \slblb^{-1}\bub} \), 
let \( A_{x} \) be the event that \( E^{*}(u, v_{1}, w_{1};\, a) \) holds
with the vertical projection of the hole ending at \( x \),
and that \( x \) is the smallest possible number with this property.
Let \( B_{x} = E^{*}(w_{1}, v_{2}, w_{2};\, x) \).

 \begin{step+}{step:compound-hole-lb.Ax-Bx}
We have \( E_{2} \spsq \bigcup_{x} (A_{x} \cap B_{x}) \).
 \end{step+}
 \begin{pproof}
\begin{sloppypar}
If for some \( x \) we have \( A_{x} \), then there is a
hole \( \Rect(\pair{t_{1}}{v_{1}},\pair{x}{w_{1}}) \) through the first wall
with the property that rectangle \( \Rect(\pair{a}{u},\pair{t_{1}}{v_{1}}) \)
contains no traps or barriers of \( \cM \) and is inner clean in \( \cM \).
Given that by assumption this rectangle contains no traps or barriers of \( \cM^{*} \),
event \( E^{*}(u, v_{1}, w_{1};\, a) \) holds.
If also \( B_{x} \) holds, then there is a rectangle
\( \Rect(\pair{x}{w_{1}},\pair{t_{2}}{v_{2}}) \) satisfying the requirements 
of \( E^{*}(w_{1}, v_{2}, w_{2};\, x) \),
and also a hole \( \Rect(\pair{t_{2}}{v_{2}},\pair{x'}{w_{2}}) \) through the second
wall.
\end{sloppypar}

Let us show that \( \pair{t_{1}}{v_{1}}\leadsto\pair{x'}{w_{2}} \), and thus 
the interval \( \rint{t_{1}}{x'} \) is a hole that
passes through the compound wall \( W \).

The reachabilies \( \pair{t_{1}}{v_{1}}\leadsto\pair{x}{w_{1}} \) 
and \( \pair{t_{2}}{v_{2}}\leadsto(x', w_{2}) \) follow by the definition of holes; 
the reachability \( \pair{x}{w_{1}}\leadsto\pair{t_{2}}{v_{2}} \)
remains to be proven.

Since the event \( B_{x} \) holds, by Lemma~\ref{lem:pre-hole}
\( \pair{x}{w_{1}} \), \( \pair{t_{2}}{v_{2}} \) satisfy the slope conditions.
Let us show that then actually \( \Rect(\pair{x}{w_{1}},\pair{t_{2}}{v_{2}}) \) is a
hop of \( \cM \): then its endpoint is reachable from
its starting point according to the reachability condition of \( \cM \).

To see that the rectangle is a hop:  the inner H-cleanness of \( \rint{x}{t_{2}} \)
in the process \( X \) follows from \( B_{x} \); the latter also implies that
there are no vertical walls in \( \rint{x}{t_{2}} \).
The inner cleanness of \( \rint{w_{1}}{v_{2}} \) in the process \( Y \) is 
implied by the fact that \( \rint{v_{1}}{w_{2}} \) is a compound wall.
The fact that \( W \) is a compound wall also implies that the interval
\( \rint{w_{1}}{v_{2}} \) contains no horizontal walls.
These facts imply the inner cleanness of the rectangle
\( \Rect(\pair{w_{1}}{x},\pair{v_{1}}{t_{2}}) \).
 \end{pproof}

It remains to lower-bound \( \Prob\bigparen{\bigcup_{x} (A_{x} \cap B_{x})} \).
For each \( x \), the events \( A_{x},B_{x} \) belong to disjoint intervals, 
and the events \( A_{x} \) are disjoint of each other.
 
 \begin{step+}{step:compound-hole-lb.Ax}
Let us lower-bound \( \sum_{x} \Prob(A_{x}) \).
 \end{step+}
 \begin{prooof}
   \begin{sloppypar}
We have, using the notation of Lemma~\ref{lem:hole-lb-2}:
 \( \sum_{x} \Prob(A_{x}) = \Prob(E^{*}(u, v_{1}, w_{1};\, a)) \). 
Lemma~\ref{lem:hole-lb-2} is applicable and we get 
\( \Prob(E^{*}(u, v_{1}, w_{1};\, a)) \ge F_{1} - U \) with
\( F_{1}=0.25 \land (v_{1}-u+1)^{\hxp} \h(\r_{1}) \),
and \( U \) coming from Lemma~\ref{lem:hole-lb-2}.
Now
\( (v_{1}-u+1)^{\hxp} \h(\r_{1}) 
\le (\slblb^{-2}\bub^{*}+1)^{\hxp} \h(\r_{1}) \)
which by assumption~\eqref{eq:s-ub} is \(  \le 0.25 \).
So the operation \( 0.25 \land \) can be deleted from \( F_{1} \):
 \begin{equation*}
   F_{1} = (v_{1}-u+1)^{\hxp} \h(\r_{1}).
 \end{equation*}
   \end{sloppypar}
 \end{prooof}

 \begin{step+}{step:compound-hole-lb.Bx}
Let us now lower-bound \( \Prob(B_{x}) \).
 \end{step+}
 \begin{prooof}
We have \( B_{x} = E^{*}(w_{1}, v_{2}, w_{2};\, x) \).
The conditions of Lemma~\ref{lem:hole-lb-2} are satisfied for
\( u = w_{1} \), \( v = v_{2} \), \( w = w_{2} \), \( a = x \).
It follows that \( \Prob(B_{x}) \ge F_{2} - U \) with
\( F_{2} = 0.25 \land (\D+1)^{\hxp} \h(\r_{2}) \),
which can again be simplified using assumption~\eqref{eq:s-ub} and
\( \D\le\f \):
 \[
  F_{2} = (\D+1)^{\hxp} \h(\r_{2}).
 \]
 \end{prooof}

 \begin{step+}{step:compound-hole-lb.combine}
Let us combine these estimates, using 
\( G = F_{1} \land F_{2} > \h(\r_{1}\lor \r_{2}) \).
 \end{step+}
 \begin{prooof}
We have
 \begin{align*}
   \Prob(E_{2}) &\ge \sum_{x} \Prob(A_{x})\Prob(B_{x})
      \ge (F_{1} - U)(F_{2} - U)
\\    &\ge F_{1} F_{2}(1 - U(1/F_{1} + 1/F_{2}))
       \ge F_{1} F_{2}(1 - 2 U/G)
\\    &= (v_{1}-u+1)^{\hxp} (\D+1)^{\hxp} \h(\r_{1}) \h(\r_{2}) 
       (1 - 2 U / G)
\\    &\ge (v_{1}-u+1)^{\hxp} (\D+1)^{\hxp} \h(\r_{1}) \h(\r_{2}) 
       (1 - 2 U/h(\r_{1}\lor \r_{2})).
 \end{align*}
 \end{prooof}
 \begin{step+}{step:compound-hole-lb.distance}
We conclude by showing \( (\D+1)\ge \lg^{i} \).
 \end{step+}
 \begin{prooof}
If \( \D=0 \) or \( 1 \) then \( i=D \), so this is true.
If \( \D>1 \) then \( i\le \log_{\lg}\D \), so even \( \D\ge\lg^{i} \).
 \end{prooof}
 \end{Proof}

The lemma below is essentially the substitution of the scale-up parameters into the 
above one.

 \begin{lemma}\label{lem:all-compound-hole-lb}
After choosing 
\( \aux_{3},\R_{1} \) sufficiently large in this order, the following holds. 
Assume that \( \cM=\cM^{k} \) is a mazery: then every
compound wall satisfies the hole lower bound,
Condition \ref{cond:distr}.\ref{i:distr.hole-lb}, provided its components
satisfy it.
 \end{lemma}

This lemma corresponds to Lemma~7.10 of~\cite{GacsWalks11}.

For the hole lower bound condition for \( \cM^{*} \), there is one more case to consider.

 \begin{lemma}\label{lem:heavy-hole-lb}
After choosing \( \aux_{3},\R_{1} \) sufficiently large in this order, 
the following holds.
Assume that \( \cM=\cM^{k} \) is a mazery: then every
wall of \( \cM^{k+1} \) that is also a heavy wall of \( \cM^{k} \)
satisfies the hole lower bound,
Condition \ref{cond:distr}.\ref{i:distr.hole-lb}.
 \end{lemma}
This lemma corresponds to Lemma~7.11 of~\cite{GacsWalks11}. 

\subsection{Auxiliary bounds}

The next lemma shows that the choice made in
Definition~\ref{def:new-ncln} satisfies the requirements.

 \begin{lemma}\label{lem:0.6}
If \( \R_{1} \) is sufficiently large then
inequality~\eqref{eq:main.big-sum} holds, moreover 
\begin{align*}
\sum_{k}\bigparen{ 2\bub_{k+1}\T^{-1}_{k} + \bub_{k+1}^{2}\tub_{k}} < 1/4.
 \end{align*}
 \end{lemma}
 \begin{proof}
The event \( \cF_{k}(\bub_{k+1}) \) says that some wall or trap of level \( k \) appears
in \( \clint{0}{\bub_{k+1}}^{2} \).
The event \( \cQ_{k+1}\setminus \cQ_{k} \) implies that a trap of level \( k \)
appears \( \clint{0}{\bub_{k+1}}^{2} \).
The probability that a wall of level \( k \) appears in \( \clint{0}{\bub_{k+1}}^{2} \)
is clearly bounded by \( 2\bub_{k+1}\T^{-1}_{k} \).
The probability that a trap of level \( k \) appears there is
bounded by \( \bub_{k+1}^{2}\tub_{k} \).
Hence \( \Prob\bigparen{ \cF_{k}(\bub_{k+1})\cup \cQ_{k+1}\setminus \cQ_{k}} \) is
bounded by \( 2\bub_{k+1}\T^{-1}_{k} + \bub_{k+1}^{2}\tub_{k} \).

The rest of the statement and its proof
correspond to Lemma~7.6 of~\cite{GacsWalks11}.
 \end{proof}

Note that for \( \R_{1} \) large enough, the relations
 \begin{align}
  \label{eq:strong-ncln-ub}
                     \bub^{*}\T^{-1} &< 0.5(0.05 - \ncln_{\w}),
\quad            \bub^{*}\T^{-1} < 0.5(0.55 - \ncln_{\t}),
\\\label{eq:strong-f-intro}
\slopeincr\slblb^{-3}\bub /\g &< 0.5(1.1/2\R_{1} - \slb_{x}),
\quad        \slopeincr\slblb^{-3}\bub /\g < 0.5(1.1\R_{1} - \slb_{y})
 \end{align}
hold for \( \cM=\cM^{1} \) as defined in Example~\ref{xmp:base}.
This is clear for~\eqref{eq:strong-ncln-ub}. 
For~\eqref{eq:strong-f-intro}, we only need the two inequalities
\( 1/40\R_{1} > \slopeincr\slblb^{-3}\bub /\g = 8\slopeincr\R_{1}^{3} \T^{-(\gxp - \bubxp)} \),
\( \R_{1}/20 > 8\slopeincr\R_{1}^{3} \T^{-(\gxp - \bubxp)} \),
both of which are satisfied if \( \R_{1} \) is large enough.

\begin{lemma}\label{lem:ncln-ub}
Suppose that the structure \( \cM=\cM^{k} \) is a mazery and it
satisfies~\eqref{eq:strong-ncln-ub} and~\eqref{eq:strong-f-intro}.
Then \( \cM^{*}=\cM^{k+1} \) also satisfies these inequalities if \( \R_{1} \) is
chosen sufficiently large (independently of \( k \)),
and also satisfies Condition \ref{cond:distr}.\ref{i:distr.ncln}.
 \end{lemma}
This lemma corresponds to Lemma~7.7 of~\cite{GacsWalks11}, and its proof is
essentially also: the changed initial values and bounds of \( \slb_{x},\slb_{y} \)
and \( \ncln_{\w},\ncln_{\t} \) do not change the arguments due to the negative
exponential dependence of their increments on \( \R_{1} \).
Recall the definition of \( \slb_{i}^{*} \) in Definition~\ref{def:new-slb}, and
the definition of \( \ncln_{i}^{*} \) in Definition~\ref{def:new-ncln}.

 \begin{proof}
Let us show first that \( \cM^{*} \) also satisfies the inequalities if \( \R_{1} \) is
chosen sufficiently large.

\begin{sloppypar}
For sufficiently large \( \R_{1} \), we 
have \( \bub^{**} (\T^{*})^{-1} < 0.5 \bub^{*} \T^{-1} \).
Indeed, this says \( \T^{(\txp\bubxp - 1)(\txp - 1)} < 0.5 \).
Hence using~\eqref{eq:strong-ncln-ub} 
and the definition of \( \ncln^{*}_{\w} \) in Definition~\ref{def:new-ncln}:
\begin{align*}
   \bub^{**} (\T^{*})^{-1} &\le 0.5\bub^{*} \T^{-1}
    \le 0.5(0.05 - \ncln_{\w}) - 0.5\bub^{*}\T^{-1}
\\           &= 0.5(0.05 - \ncln^{*}_{\w}).
  \end{align*}
This is the first inequality of~\eqref{eq:strong-ncln-ub} for \( \cM^{*} \).
The second one is proved the same way.
To verify Condition \ref{cond:distr}.\ref{i:distr.ncln} for
\( \cM^{*} \), recall Definition~\ref{def:new-ncln} of \( \ncln^{*}_{i} \).
For inequality~\eqref{eq:ncln.1dim}, for an upper bound on 
the conditional probability that a point \( a \) of the line is 
strongly clean in \( \cM \) but not in \( \cM^{*} \) let us use
 \begin{equation*}
  (2\f/3+\bub)\T^{-1},
 \end{equation*}
which upper-bounds the probability that a 
vertical barrier of \( \cM \) starts in \( \rint{a - \f/3-\bub}{a + \f/3} \).
This can be upper-bounded by \( \f\T^{-1} < \bub^{*}\T^{-1} \) 
by~\eqref{eq:bub-g-f} for sufficiently large \( \R_{1} \).
Hence an upper bound on the conditional probability of not strong cleanness in
\( \cM^{*} \) is \( \ncln_{\w}+\bub^{*} \T^{-1}=\ncln_{\w}^{*} \) as required,
due to Definition~\ref{def:new-ncln}.
  \end{sloppypar}

  \begin{sloppypar}
For the other inequalities in Condition \ref{cond:distr}.\ref{i:distr.ncln},
consider a rectangle \( Q=\Rect^{\to}(u,v) \) and fix \( Y=y \).
The conditional probability that a point \( u \) is trap-clean in \( Q \) for \( \cM \) but
not for \( \cM^{*} \) is upper-bounded by the probability of the appearance of a
trap of \( \cM \) within a distance \( \g \) of point \( u \) in \( Q \).
There are at most \( \g^{2} \) positions for the trap, so a bound is
\[
    \g^{2}\tub =  \T^{2\gxp - \tubxp} < \T^{\txp \bubxp - 1},
\]
where the last inequality follows from~\eqref{eq:trap-xp}.
We conclude the same way for the first inequality.
The argument for the other inequalities in
Condition \ref{cond:distr}.\ref{i:distr.ncln} is identical.
  \end{sloppypar}

For the first inequality of~\eqref{eq:strong-f-intro}, the scale-up definition
Definition~\ref{def:new-slb} says
  \( \slb^{*}_{x} - \slb_{x} = \slopeincr\slblb^{-3}\bub/\g \). 
The inequality \( \bub^{*}/\g^{*} < 0.5\bub/\g \) 
is guaranteed if \( \R_{1} \) is large. 
From here, we can conclude the proof as for \( \ncln_{i} \); similarly for \( \slb_{y} \).
 \end{proof}

 \section{The approximation lemma}
\label{sec:approx-proof}

The crucial combinatorial step in proving the main lemma is the following.

 \begin{lemma}[Approximation]\label{lem:approx}
The reachability condition, Condition \ref{cond:distr}.\ref{i:distr.reachable},
holds for \( \cM^{*} \) if \( \R_{1} \) is sufficiently large.
 \end{lemma}

The present section is taken up by the proof of this lemma.

Recall that we are considering a bottom-open or left-open or closed rectangle \( Q \)
with starting point \( u=\pair{u_{0}}{u_{1}} \) and endpoint \( v=\pair{v_{0}}{v_{1}} \)
with \( u_{\d}<v_{\d} \), \( \d=0,1 \) with the property that there is a (non-integer)
point \( v'=\pair{v'_{0}}{v'_{1}} \) with \( 0\le v_{0}-v'_{0},v_{1}-v'_{1}<1 \) 
such that
 \begin{align}\label{eq:slope-req-star}
   \slb^{*}_{x}\le \slope(u,v')\le (\slb^{*}_{y})^{-1}.
 \end{align}
We require \( Q \) to be a hop of \( \cM^{*} \).
Thus, the points \( u, v \) are clean for \( \cM^{*} \) in \( Q \), and \( Q \)
contains no traps or walls of \( \cM^{*} \).
We have to show \( u\leadsto v \).
Assume 
 \[
  Q = I_{0} \times I_{1} = \Rect^{\eps}(u, v)  
 \]
where \( \eps=\rightarrow,\uparrow \) or nothing.

 \subsection{Walls and trap covers}

Let us determine the properties of the set of walls in \( Q \).

 \begin{lemma}\label{lem:the-grate}
Under conditions of Lemma~\ref{lem:approx}, with the notation given in the
discussion after the lemma, the following holds.
 \begin{alphenum}

  \item\label{i:grate.pre-hop}
For \( \d=0,1 \), for some \( n_{\d} \ge 0 \), there
is a sequence \( W_{d,1},\dots,W_{d,n_{d}} \) of dominant 
light neighbor walls of \( \cM \) separated from each other 
by external hops of \( \cM \) of size \( >\f \), and from the ends of \( I_{\d} \)
(if \( n_{\d}>0 \)) by hops of \( \cM \) of size  \( \ge \f/3 \).

  \item\label{i:grate.hole}
For every (horizontal) wall \( W_{0,i} \) of \( \cM \) occurring in \( I_{1} \),
for every subinterval \( J \) of \( I_{0} \) of size \( \g \) such that \( J \) is at a
distance \( \ge\g+7\bub \) from the ends of \( I_{0} \), there is an outer 
rightward clean hole fitting \( W_{0,i} \), with endpoints at a distance of at
least \( \bub \) from the endpoints of \( J \).
The same holds if we interchange vertical and horizontal.
 \end{alphenum}
 \end{lemma}

 \begin{proof}
This is a direct consequence of Lemmas~\ref{lem:new-hop}
and~\ref{lem:if-not-missing-hole}.
The vertical cleanness needed in the outer rightward cleanness of the hole 
through \( W_{0,i} \) follows from part~\eqref{i:grate.pre-hop}.
 \end{proof}

From now on, in this proof, whenever we mention a \df{wall} we mean one of
the walls \( W_{\d, i} \), and whenever we mention a trap then, unless said
otherwise, we mean only traps of \( \cM \) entirely within \( Q \) and
not intersecting any of these walls.
Let us limit the places where traps can appear in \( Q \).

 \begin{definition}[Trap cover]
A set of the form \( I_{0} \times J \)
with \( |J| \le 4 \bub \) containing the starting point of a
trap of \( \cM \) will be called a \df{horizontal trap cover}.
Vertical trap covers are defined similarly.
 \end{definition}

In the following lemma, when
we talk about the distance between two traps, we mean the distance between
their starting points.

 \begin{lemma}[Trap cover]\label{lem:trap-cover}
Let \( T_{1} \) be a trap of \( \cM \) contained in \( Q \).
Then there is a horizontal or vertical trap cover \( U \spsq T_{1} \) such that
the starting point of
every other trap in \( Q \) is either contained in \( U \) or is
at least at a distance \( \f - \bub \) from \( T_{1} \).
If the trap cover is vertical, it intersects none of the vertical walls
\( W_{0, i} \); if it is horizontal, it intersects none of the horizontal
walls \( W_{1,j} \).
 \end{lemma}

This lemma corresponds to Lemma~8.3 of~\cite{GacsWalks11}.

Let us measure distances from the line defined by the points \( u,v' \).

\begin{definition}[Relations to the diagonal]\label{def:diagonal}
Define, for a point \( a = \pair{a_{0}}{a_{1}} \):
 \[
   d_{u,v'}(a)=d(a) = (a_{1}-u_{1}) - \slope(u,v')(a_{0} - u_{0})
 \]
to be the distance of \( a \) above the line of \( u,v' \), then for
\( w=\pair{x}{y} \), \( w'=\pair{x'}{y'} \):
 \begin{equation*}
   \begin{aligned}
            d(w')-d(w) &= y'-y - \slope(u,v')(x'-x),
\\            |d(w')-d(w)| &\le |y'-y|+|x'-x|/\slb_{y}.
 \end{aligned}
 \end{equation*}
We define the strip
 \[
 C^{\eps}(u, v', h_{1}, h_{2}) = \setof{w \in \Rect^{\eps}(u,v) :
   h_{1} < d_{u,v'}(w) \le h_{2}},
 \]
a channel of vertical width \( h_{2}-h_{1} \) in \( \Rect^{\eps}(u,v) \),
parallel to line of \( u,v' \) . 
\end{definition}

 \begin{lemma}\label{lem:approx.no-walls}
Assume that points \( u, v \) are clean for \( \cM \) in \( Q=\Rect^{\eps}(u,v) \),
with 
 \[
   \slb_{x}+4\bub/\g \le \slope(u,v') \le 1/(\slb_{y} + 4\slblb^{-2}\bub/\g),
 \]
where \( v' \) relates to \( v \) as above.
If \( C = C^{\eps}(u, v', -\g, \g) \) contains 
no traps or walls of \( \cM \) then \( u \leadsto v \).
(By \( C \) not containing walls we mean that its projections don't.)
 \end{lemma}
This lemma corresponds to Lemma~8.4 of~\cite{GacsWalks11}, and 
Figure~23 there illustrates the proof.
 \begin{proof}
Let \( \mu=\slope(u,v') \).
If \( |I_{0}| < \g \) then \( C=Q \), so
there is no trap or wall in \( Q \), therefore \( Q \) is a
hop, and we are done via Condition \ref{cond:distr}.\ref{i:distr.reachable} for
\( \cM \). 
Suppose \( |I_{0}| \ge \g \).
Let 
 \[
      n = \Cei{\frac{|I_{0}|}{0.9 \g}},
\quad h = \frac{|I_{0}|}{n}.
 \]
Then \( \g / 2 \le h \le 0.9 \g \).
Indeed, the proof of the second inequality is immediate.
For the first one, if \( n\le 2 \), we have \( \g \le |I_{0}| = n h \le 2 h \), and
for \( n\ge 3 \):
 \begin{align*}
   \frac{|I_{0}|}{0.9\g}&\ge n-1,
\\ |I_{0}|/n &\ge (1-1/n) 0.9\g \ge 0.6\g.   
 \end{align*}
For \( i = 1, 2, \dots, n-1 \), let
 \[
      a_{i} = u_{0} + i h,
\quad b_{i} = u_{1} + i h\cdot \mu,
\quad w_{i} = \pair{a_{i}}{b_{i}},
\quad S_{i} = w_{i} + \clint{-\bub}{2\bub}^{2}.
 \]
Let us show \( S_{i}\sbsq C \).
For all elements \( w \) of \( S_{i} \), we have 
\( |d(w)| \le 2(1+1/\slb_{y})\bub \), and we
know \( 2(1+1/\slb_{y})\bub < \g \) if \( \R_{1} \) is sufficiently large.
To see \( S_{i} \sbsq \Rect^{\eps}(u,v) \),
we need (from the worst case \( i=n-1 \)) 
\( \mu h > 2\bub \).
Using the above and the assumptions of the lemma: 
 \[
   \frac{2\bub}{h} \le \frac{2\bub}{\g/2}
   = 4\bub/\g \le \mu.
 \]
By Remark~\ref{rem:distr}.\ref{i:middle-third},
there is a clean point \( w'_{i} = \pair{a'_{i}}{b'_{i}} \) in the middle third 
\( w_{i}+\clint{0}{\bub}^{2} \) of \( S_{i} \).
Let \( w'_{0} = u \), \( w'_{n} = v' \).
By their definition, each rectangle \( \Rect^{\eps}(w'_{i}, w'_{i+1}) \)
rises by at most \( < \mu(0.9 \g + \bub)+\bub < \g \),
above or below the diagonal, hence
falls into the channel \( C \) and is consequently trap-free.

If \( \slb_{x}\le\slope(w'_{i}, w'_{i+1})\le 1/\slb_{y} \) this will imply
\( w'_{i}\leadsto w'_{i+1} \) for \( i<n-1 \), and \( w'_{n-1}\leadsto v \).
Let \( \mu'=\slope(w'_{i}, w'_{i+1}) \).
We know already \( \mu\ge \slb_{x} + 4\bub/\g \) and 
\( 1/\mu \ge \slb_{y}+4\slblb^{-2}\bub/\g \). 
It is sufficient to show \( \mu-\mu'\le 4\bub/\g \) and
 \( 1/\mu - 1/\mu'\le 4\slblb^{-2}\bub/\g \).

The distance from \( w'_{i} \) to \( w'_{i+1} \) is between
\( h-\bub \) and \( h+\bub \) in the \( x \) coordinate and between
\( \mu h -\bub \) and \( \mu h+\bub \) in the \( y \) coordinate. 
We have
  \begin{align*}
 \mu-\mu' \le \mu-\frac{\mu h - \bub}{h+\bub} =\frac{(\mu+1)\bub}{h+\bub} 
  \le \frac{(\mu+1)\bub}{\g/2+\bub} \le 4\bub/\g.
  \end{align*}
Similarly
  \begin{align*}
    \frac{1}{\mu}-\frac{1}{\mu'} \le \frac{1}{\mu}-\frac{h-\bub}{\mu h+\bub}
   =\frac{(\mu+1)\bub}{\mu (\mu h+\bub)} 
   \le \frac{(\mu+1)\bub}{\mu^{2}\g/2}.
 \end{align*}
The condition of the lemma implies \( \slblb\le \mu\le 1 \), and this
implies that the last expression is less than 
\( 4\bub/\mu^{2}\g \le 4\slblb^{-2}\bub/\g \).
 \end{proof}

We introduce particular strips around the diagonal.

 \begin{definition} Let \( \gf=(\g\f)^{1/2} \), \( C =  C^{\eps}(u, v', -3\gf, 3\gf) \),
where \( v' \) is defined as above.
 \end{definition}

Let us introduce the system of walls and trap covers we will have to
overcome.

 \begin{definition}
Let us define a sequence of trap covers \( U_{1}, U_{2},\dots \) as follows.
If some trap \( T_{1} \) is in \( C \), then let
\( U_{1} \) be a (horizontal or vertical) trap cover covering it according to
Lemma~\ref{lem:trap-cover}.
If \( U_{i} \) has been defined already and there is a trap \( T_{i+1} \) in
\( C \) not covered by \( \bigcup_{j \le i} U_{j} \) then let \( U_{i+1} \)
be a trap cover covering this new trap.
To each trap cover \( U_{i} \) we assign a real number \( a_{i} \) as follows.
Let \( \pair{a_{i}}{a'_{i}} \) be the intersection of the diagonal of \( Q \) and the
left or bottom edge of \( U_{i} \) (if \( U_{i} \) is vertical or horizontal
respectively).
Let \( \pair{b_{i}}{b'_{i}} \) be the intersection of the diagonal and the left edge
of the vertical wall \( W_{0,i} \) introduced in Lemma~\ref{lem:the-grate},
and let \( \pair{c'_{i}}{c_{i}} \) be the intersection of the diagonal and the
bottom edge of the horizontal wall \( W_{1,i} \).
Let us define the finite set
 \begin{equation*}
   \set{s_{1}, s_{2},\dots} 
= \set{a_{1}, a_{2}, \dots} \cup \set{b_{1}, b_{2}, \dots} \cup 
  \set{c'_{1}, c'_{2}, \dots}
  \end{equation*}
where \( s_{i} \le s_{i+1} \).

We will call the objects (trap covers or walls) belonging to the points
\( s_{i} \) our \df{obstacles}.
 \end{definition}

 \begin{lemma}\label{lem:one-of-three}
If \( s_{i},s_{j} \) belong to the same obstacle category among the three
(horizontal wall, vertical wall, trap cover) then
\( |s_{i}-s_{j}| \ge 0.75\f \) for \( \R_{1} \) sufficiently large.
\end{lemma}
This lemma corresponds to Lemma~8.5 of~\cite{GacsWalks11}. 

It follows that for every \( i \) at least one of the three numbers
\( (s_{i+1} - s_{i}) \), \( (s_{i+2} - s_{i+1}) \), \( (s_{i+3} - s_{i+2}) \) is larger
than \( 0.25\f \).



 \subsection{Passing through the obstacles}

The remark after Lemma~\ref{lem:one-of-three} allows us to break up the sequence
of obstacles into groups of size at most three, which can be dealt with
separately.
So the main burden of the proof of the Approximation Lemma is carried by
following lemma.

 \begin{lemma}\label{lem:three}
There is a constant \( \slopeincr \) with the following properties.
Let \( u, v \) be points with
 \begin{equation}\label{eq:triangle-cnd.three}
   \begin{aligned}
     \slb_{x} + (\slopeincr-1)\slblb^{-3}\bub/\g &\le \slope(u,v'),
\\   \slb_{y} + (\slopeincr-1)\slblb^{-3}\bub/\g &\le 1/\slope(u,v'),
   \end{aligned}
 \end{equation}
where \( v' \) is related to \( v \) as above.
Assume that the set \( \set{s_{1}, s_{2}, \dots} \) defined above 
consists of at most three elements, with
the consecutive elements less than \( 0.25\f \) apart.
Assume also 
 \begin{equation}\label{eq:far-from-ends}
   v_{0} - s_{i},\; s_{i} - u_{0} \ge 0.1\f.
 \end{equation}
Then if \(\Rect^{\rightarrow}(u, v)\) or \(\Rect^{\uparrow}(u, v)\) is a 
hop of \( \cM^{*} \) then \( u\leadsto v \).
 \end{lemma}
\begin{Proof}
\begin{sloppypar}
Let \( \mu=\slope(u,v') \), and note that the conditions imply \( \mu\le 1 \).
We can assume without loss of generality that there are indeed three
points \( s_{1} \), \( s_{2} \), \( s_{3} \).
By Lemma~\ref{lem:one-of-three}, they must then come from three obstacles of
different categories: \( \set{s_{1}, s_{2}, s_{3}} \) = \( \set{a, b, c'} \) where
\( b \) comes from a vertical wall, \( c' \) from a horizontal wall, 
and \( a \) from a trap cover.
There is a number of cases.
\end{sloppypar}

If the index \( i\in\{1,2,3\} \) 
of a trap cover is adjacent to the index of a wall of the same
orientation, then this pair will be called a \df{parallel pair}.
A parallel pair is either horizontal or vertical.
It will be called a \df{trap-wall pair} if
the trap cover comes first, and the \df{wall-trap pair} if the wall comes
first.

We will call an obstacle \( i \) \df{free}, if it is not part of a 
parallel pair.
Consider the three disjoint channels 
 \[
  C(u, v', K-\gf, K +\gf), \txt{ for } K=-2\gf,\, 0,\, 2\gf.
 \]
The three lines (bottom or left edges) of the trap covers or walls
corresponding to \( s_{1},s_{2},s_{3} \) can intersect in at most two places,
so at least one of the above channels does not contain such an
intersection.
Let \( K \) belong to such a channel.
Its middle is the line \( C(u,v',K,K) \).
For \( i \in\{1,2,3\} \), let
 \begin{align*}
 w_{i}=\pair{x_{i}}{y_{i}}
 \end{align*}
be the intersection point of the starting edge of obstacle \( i \) with this line.
These points will guide us to define the rather close points
 \[
w'_{i} = \pair{x'_{i}}{y'_{i}}, 
\quad  w''_{i} = \pair{x''_{i}}{y''_{i}}
 \]
in the channel \( C(u, v', K-\gf, K+\gf) \) through which an actual path will
go.
Not all these points will be defined, but they will
always be defined if \( i \) is free.
Their role in this case is the following:
\( w'_{i} \) and \( w''_{i} \) are points on the two sides of the trap cover or
wall with \( w'_{i}\leadsto w''_{i} \).
We will have
 \begin{align}\label{eq:diag-close}
   |x-x_{i}|+ |y-y_{i}| = O(\slblb^{-1}\g)
 \end{align}
for \( x=x'_{i},x''_{i} \) and \( y=y'_{i},y''_{i} \).

We will make use of the following relation 
for arbitrary \( a=\pair{a_{0}}{a_{1}} \), \( b=\pair{b_{0}}{b_{1}} \):
 \begin{equation}\label{eq:slope-by-d}
   \slope(a, b) = \mu + \frac{d(b)-d(a)}{b_{0}-a_{0}}.
 \end{equation}

For the analysis that follows, note that all points within distance \( \gf/2 \) of
any points \( w_{i} \) are contained in the channel \( C \), and hence
also in the rectangle \( Q \).

The following general remark will also be used several times below.
Suppose that for one of the (say, vertical) trap covers with starting point \( x_{i} \),
we determine that the rectangle 
\( \clint{x_{i}}{x_{i}+5\bub} \times I \) 
intersecting the channel \( C \), where \( |I|<\gf \), contains no trap.
Then the much largest rectangle \( \clint{x_{i}-\f}{x_{i}+\f}\times I \) 
contains no trap either.
Indeed, there is a trap somewhere in the intersection of the channel with the trap cover
\( C \) (this is why the trap cover is needed), and then 
the trap cover property implies that there is no other trap outside the trap cover
within distance \( \f\gg\gf \) of this trap.

  \begin{step+}{step:three.x-free.trapc.vert}
 Consider crossing a free vertical trap cover.
  \end{step+}
 \begin{prooof}
Recall the definition of \( \L_{2} \) in~\eqref{eq:Lt}.
We apply Lemma~\ref{lem:if-correl} to vertical correlated traps
\( J\times I' \), with 
\( J=\clint{x_{i}}{x_{i}+5\bub} \), \( I'=\clint{y_{i}}{y_{i}+\L_{2}} \). 
The lemma is applicable since
\( w_{i}\in C(u, v', K-\gf, K+\gf) \) implies
\( u_{1}< y_{i}-\L_{2}-7\bub < y_{i}+2\L_{2}+7\bub < v_{1} \).
Indeed, formula~\eqref{eq:far-from-ends} implies,
using~\eqref{eq:triangle-cnd.three}: 
 \begin{align*}
 y_{i} > u_{1} + 0.1\mu\f \ge u_{1}+7\bub+\L_{2}
 \end{align*}
for sufficiently large \( \R_{1} \), using \( \L_{2}\ll\f \).
The inequality about \( v_{1} \) is similar, using the other inequality
of~\eqref{eq:far-from-ends}.

Lemma~\ref{lem:if-correl} implies that there is a region 
\( \clint{x_{i}}{x_{i}+5\bub} \times \clint{y}{y+2.2\slblb^{-1}\g} \) containing no traps,
with \( \lint{y}{y+2.2\slblb^{-1}\g} \sbsq \lint{y_{i}}{y_{i} + \L_{2}} \).
Thus, there is a \( y \) in \( \lint{y_{i}}{y_{i}+\L_{2}-2.2\slblb^{-1}\g} \)
such that \( \clint{x_{i}}{x_{i}+5\bub} \times \clint{y}{y+2.2\slblb^{-1}\g} \) contains no
traps.
(In the present proof, all other arguments finding a region with no traps
in trap covers are analogous, so we will not mention
Lemma~\ref{lem:if-correl} explicitly again.)
Since all nearby traps must start in a trap cover, the region
\( \clint{x_{i}-2\bub}{x_{i}+\g} \times \clint{y}{y+2.2\slblb^{-1}\g} \)
contains no trap either. 
Thus there are clean points \( w'_{i} \) in 
\( \pair{x_{i}-\bub}{y+\bub} + \clint{0}{\bub}^{2} \) and
\( w''_{i} \) in \( \pair{x_{i}+\g-2\bub}{y+\slb_{x}\g+\bub}
+\clint{0}{\bub}^{2} \).
Let us estimate \( \slope(w'_{i}, w''_{i}) \).
We have
\begin{equation}\label{eq:vertical-trapcover-forward}  
 \begin{aligned}
  \g-2\bub &\le x''_{i}-x'_{i}\le \g,
\\  \slb_{x}\g &\le y''_{i}-y'_{i}\le \slb_{x}\g+2\bub,
\\ \slb_{x} &\le\slope(w'_{i},w''_{i})\le
\frac{\slb_{x}\g+2\bub}{\g-2\bub}\le\slb_{x}+\frac{4\bub}{\g-4\bub}
\\ &\le \slb_{x}^{*}\le 1/\slb_{y}
 \end{aligned}
\end{equation}
if \( \R_{1} \) is large, where we used Definition~\ref{def:new-slb}
and~\eqref{eq:slb-ub}.
So the pair \( w'_{i},w''_{i} \) satisfies the slope conditions.
The rectangle between them is also trap-free, due to
\( \slb_{x}\g+2\bub\le 2\g \), hence \( w'_{i}\leadsto w''_{i} \).

The point \( w'_{i} \) is before the trap cover defined by \( w_{i} \), while
 \( w''_{i} \) is after.
Their definition certainly implies the relations~\eqref{eq:diag-close}.

 \end{prooof} 

 \begin{step+}{step:three.x-free.trapc.horiz}
 Consider crossing a free horizontal trap cover.
 \end{step+}
 \begin{prooof}
There is an \( x \) in \( \lint{x_{i}-\L_{2}}{x_{i}-7\bub} \) such that
\( \clint{x}{x+2.2\slblb^{-1}\g} \times \clint{y_{i}}{y_{i}+5\bub}\) contains no trap.
Thus there are clean points 
\( w'_{i} \) in \( \pair{x+\bub}{y_{i}-\bub} + \clint{0}{\bub}^{2}\) and
\( w''_{i} \) in \( \pair{x+\slb_{x}^{-1}\g}{y_{i}+\g} + \clint{0}{\bub}^{2}\). 
Now estimates similar to~\eqref{eq:vertical-trapcover-forward} hold again, so 
\( w'_{i}\leadsto w''_{i} \).
The point \( w'_{i} \) is before the trap cover defined by \( w_{i} \), while
 \( w''_{i} \) is after.
Their definition implies the relations~\eqref{eq:diag-close}.
 \end{prooof} 

 \begin{step+}{step:three.x-free.wall.vert}
 Consider crossing a free vertical wall.
 \end{step+}
 \begin{prooof} 
Let us apply Lemma~\ref{lem:the-grate}\eqref{i:grate.hole}, with
\( I'=\clint{y_{i}}{y_{i}+\g} \). 
The lemma is applicable since by
\( w_{i}\in C(u, v', K-\gf, K+\gf) \) we
have \( u_{1} \le y_{i}-\g-7\bub < y_{i}+2\g+7\bub<v_{1} \).
It implies that our wall contains an outer upward clean
hole \( \rint{y'_{i}}{y''_{i}} \sbsq y_{i} + \rint{\bub}{\g-\bub} \) 
passing through it.
(In the present proof, all other arguments finding a hole through walls
are analogous, so we will not mention
Lemma~\ref{lem:the-grate}\eqref{i:grate.hole} explicitly again.)
Let \( w'_{i} = \pair{x_{i}}{y'_{i}} \), and let \( w''_{i}=\pair{x''_{i}}{y''_{i}} \)
be the point on the other side of the wall reachable from \( w'_{i} \).
This definition implies the relations~\eqref{eq:diag-close}.
 \end{prooof} 

 \begin{step+}{step:three.x-free.wall.horiz}
Consider crossing a free horizontal wall.
 \end{step+}
 \begin{prooof} 
Similarly to above, this wall contains an outer rightwards clean hole 
\( \rint{x'_{i}}{x''_{i}} \sbsq x_{i} + \rint{-\g+\bub}{-\bub} \) 
passing through it.
Let \( w'_{i} = \pair{x'_{i}}{y_{i}} \) and let \( w''_{i}=\pair{x''_{i}}{y''_{i}} \) 
be the point on the other side of the wall reachable from \( w'_{i} \).
This definition implies the relations~\eqref{eq:diag-close}.
 \end{prooof} 

For a trap-wall or wall-trap pair,
we first find a big enough hole in the trap cover, and
then locate a hole in the wall that allows to pass
through the big hole of the trap cover.
There are cases according to whether we have a trap-wall pair or a
wall-trap pair, and whether it is vertical or horizontal, but the results
are all similar.
Figure~24 of~\cite{GacsWalks11} illustrates the similar construction in that paper.



  \begin{step+}{step:three.x-bound.trapc.vert}
 Consider crossing a vertical trap-wall pair \( \pair{i}{i+1} \).
  \end{step+}
 \begin{prooof}
Recall \( x_{i}=s_{i} \), \( x_{i+1}=s_{i+1} \).
Let us define \( x=x_{i}-\g \).
Find a \( y^{(1)} \) in \( \lint{y_{i}}{y_{i}+\L_{2}-2.2\slblb^{-1}\g} \) such that the region
\( \clint{x_{i}}{x_{i+1}} \times \clint{y^{(1)}}{y^{(1)}+2.2\slblb^{-1}\g}\cap C \)
contains no trap.

Let \( \tilde w=\pair{x_{i+1}}{\tilde y} \) be defined by 
\( \tilde y= y^{(1)} + \mu(x_{i+1}-x_{i}) + 1.1\slblb^{-1}\g \).
Thus, it is the point on the left edge of the wall if we intersect
it with a slope \( \mu \) line from \( \pair{x_{i}}{y^{(1)}} \) and then move up
\( 1.1\slblb^{-1}\g \).
Similarly to the forward crossing in Part~\ref{step:three.x-free.wall.vert}, the
vertical wall   
starting at \( x_{i+1} \) is passed through by an outer upward clean hole 
\( \rint{y'_{i+1}}{y''_{i+1}} \sbsq \tilde y + \rint{\bub}{\g-\bub} \).
Let \( w'_{i+1} = \pair{x_{i+1}}{y'_{i+1}} \), and let \( w''_{i+1}=\pair{x''_{i+1}}{y''_{i+1}} \) 
be the point on the other side of the wall reachable from \( w'_{i+1} \).
Define the line \( E \) of slope \( \mu \) going through the point \( w'_{i+1} \). 
Let \( w = \pair{x}{y^{(2)}} \) be the intersection of \( E \) with the vertical
line defined by \( x \), then \( y^{(2)}= y'_{i+1} - \mu(x_{i+1}-x) \).
The channel of (vertical) width \( 2.2\g \) around the  line \( E \) intersects the trap
cover  in a trap-free interval (that is smallest rectangle 
containing this intersection is trap-free).

There is a clean point \( w'_{i}\in \pair{x-\bub}{y^{(2)}}+\clint{0}{\bub}^{2} \).
(Point \( w''_{i} \) is not needed.)
We have
 \begin{align}
\label{eq:bound-d}   -\bub &\le d(w'_{i})-d(w'_{i+1}) \le \bub.
 \end{align}
\begin{sloppypar}
The relation~\eqref{eq:diag-close} is  easy to prove.
Let us show \( w'_{i} \leadsto w'_{i+1} \).
Given the trap-freeness of the channel mentioned above, it is easy to see that
the channel \( C^{\eps}(w'_{i},w'_{i+1},-\g,\g) \) is also trap-free.
We can apply Lemma~\ref{lem:approx.no-walls} after checking
its slope condition.
We get using~\eqref{eq:slope-by-d}, \eqref{eq:bound-d} and~\( x'_{i+1}-x\ge\g \):
 \begin{align*}
  \mu-\bub/\g\le \slope(w'_{i},w'_{i+1})\le\mu+\bub/\g.
 \end{align*}
\end{sloppypar}
 \end{prooof} 

 \begin{step+}{step:three.x-bound.trapc.horiz}
 Consider crossing a horizontal trap-wall pair \( \pair{i}{i+1} \).
 \end{step+}
 \begin{prooof} 
Let us define \( y=y_{i}-\g \).
There is an \( x^{(1)} \) in \( \lint{x_{i}}{x_{i}+\L_{2}-2.2\slblb^{-1}\g} \)
such that the region 
\( \clint{x^{(1)}}{x^{(1)}+2.2\slblb^{-1}\g} \times \clint{y_{i}}{y_{i+1}}\cap C \)
contains no trap.
Let \( \tilde w=\pair{\tilde x}{y_{i+1}} \) be defined by 
\( \tilde x= x^{(1)} + \mu^{-1}(y_{i+1}-y_{i}) + 1.1\slblb^{-1}\g \).
The horizontal wall   
starting at \( y_{i+1} \) is passed through by an outer rightward clean hole 
\( \rint{x'_{i+1}}{x''_{i+1}} \sbsq \tilde x + \rint{\bub}{\g-\bub} \).
Let \( w'_{i+1} = \pair{x'_{i+1}}{y_{i+1}} \), 
and \( w''_{i+1}=\pair{x''_{i+1}}{y''_{i+1}} \).
Define the line \( E \) of slope \( \mu \) going through the point \( w'_{i+1} \). 
Let \( w = \pair{x^{(2)}}{y} \) be the intersection of \( E \) with the horizontal
line defined by \( y \), then
\( x^{(2)}= x'_{i+1} - \mu^{-1}(y_{i+1}-y) \).
The channel of horizontal width \( 2.2\mu^{-1}\g \) and therefore vertical width
\( 2.2\g \) around the  line \( E \) intersects the trap cover  in a trap-free
interval. 
There is a clean point \( w'_{i}\in \pair{x^{(2)}}{y-\bub}+\clint{0}{\bub}^{2} \).
The proof of~\eqref{eq:diag-close} and \( w'_{i} \leadsto w'_{i+1} \) is similar
to the one for the vertical trap-wall pair.
 \end{prooof} 

 \begin{step+}{step:three.x-bound.wall.vert}
 Consider crossing a vertical wall-trap pair \( \pair{i-1}{i} \).
 \end{step+}
 \begin{prooof} 
This part is somewhat similar to Part~\ref{step:three.x-bound.trapc.vert}: we
are again starting the construction at the trap cover.

Let us define \( x=x_{i}+\g \).
Find a \( y^{(1)} \) in \( \lint{y_{i}}{y_{i}+\L_{2}-2.2\slblb^{-1}\g} \) such that the region
\( \clint{x_{i}}{x_{i}+5\bub} \times \clint{y^{(1)}}{y^{(1)}+2.2\slblb^{-1}\g}\cap C \)
contains no trap.
Let \( \tilde w =\pair{x_{i-1}}{\tilde y} \) be defined by 
\( y_{i-1}= y^{(1)} - \mu(x_{i}-x_{i-1}) + 1.1\slblb^{-1}\g \).
The vertical wall   
starting at \( x_{i-1} \) is passed through by an outer upward clean hole 
\( \rint{y'_{i-1}}{y''_{i-1}} \sbsq \tilde y + \rint{\bub}{\g-\bub} \).
We define \( w'_{i-1} \), and \( w''_{i-1} \) accordingly.
Define the line \( E \) of slope \( \mu \) going through the point \( w''_{i-1} \). 
Let \( w = \pair{x}{y^{(2)}} \) be the intersection of \( E \) with the vertical
line defined by \( x \), then \( y^{(2)}= y''_{i-1} + \mu(x-x''_{i-1}) \).
The channel of (vertical) width \( 2.2\g \) around the  line \( E \) intersects the trap
cover  in a trap-free interval.
There is a clean point \( w''_{i}\in \pair{x}{y^{(2)}}+\clint{0}{\bub}^{2} \).
The proof of~\eqref{eq:diag-close} and \( w''_{i-1} \leadsto w''_{i} \) is similar
to the corresponding proof for the vertical trap-wall pair.
 \end{prooof} 

 \begin{step+}{step:three.x-bound.wall.horiz}
Consider crossing a horizontal wall-trap pair \( \pair{i-1}{i} \). 
 \end{step+}
 \begin{prooof} 
This part is somewhat similar to Parts~\ref{step:three.x-bound.trapc.horiz}
and~\ref{step:three.x-bound.wall.vert}.
Let us define \( y=y_{i}+\g \).
There is an \( x^{(1)} \) in \( \lint{x_{i}}{x_{i}+\L_{2}-2.2\slblb^{-1}\g} \)
such that the region 
\( \clint{x^{(1)}}{x^{(1)}+2.2\slblb^{-1}\g} \times \clint{y_{i}}{y_{i}+5\bub}\cap C \)
contains no trap.
Let \( \tilde w =\pair{\tilde x}{y_{i-1}} \) be defined by 
\( \tilde x= x^{(1)} - \mu^{-1}(y_{i}-y_{i-1}) + 1.1\mu^{-1}\g \).
The wall starting at \( y_{i-1} \) contains an outer rightward clean hole 
\( \rint{x'_{i-1}}{x''_{i-1}} \sbsq \tilde x + \rint{\bub}{\g-\bub} \) 
passing through it.
We define \( w'_{i-1} \), \( w''_{i-1} \) accordingly.
Define the line \( E \) of slope \( \mu \) going through the 
point \( w''_{i-1} \). 
The point \( x^{(2)}= x''_{i-1} + \mu^{-1}(y-y''_{i-1}) \) is its intersection with the
horizontal line defined by \( y \).
The channel of horizontal width \( 2.2\mu^{-1}\g \) and therefore vertical width
\( 2.2\g \) around the  line \( E \) intersects the trap cover  in a trap-free
interval. 
There is a clean point \( w''_{i}\in \pair{x^{(2)}}{y}+\clint{0}{\bub}^{2} \).
The proof of~\eqref{eq:diag-close} and \( w''_{i-1} \leadsto w''_{i} \) is similar
to the corresponding proof for the vertical trap-wall pair.
 \end{prooof} 

 \begin{step+}{step:three.bound}
 We have \( u\leadsto v \).
 \end{step+}
\begin{pproof}
If there is no parallel pair then \( w'_{i}\leadsto w''_{i} \) is proven
for \( i=1,2,3 \).
Suppose that there is a parallel pair.
If it is a trap-wall pair \( \pair{i}{i+1} \), then
instead of \( w'_{i}\leadsto w''_{i} \) we proved 
\( w'_{i}\leadsto w'_{i+1} \);
if it is a wall-trap pair \( \pair{i-1}{i} \), then
instead of \( w''_{i-1}\leadsto w'_{i} \) we proved 
\( w''_{i-1}\leadsto w''_{i} \).

In both cases, it remains to prove
\( w''_{i}\leadsto w'_{i+1} \) whenever \( \pair{i}{i+1} \) is not a
parallel pair and \( i=1,2 \), further  
\( u\leadsto w'_{1} \), \( w''_{3}\leadsto v \). 

The rectangle \( \Rect(w''_{i},w'_{i+1}) \) is a hop by definition.
We just need to check that it satisfies the slope condition of \( \cM \).
Since \( \pair{i}{i+1} \) is not a parallel pair they intersect, 
and by the choice of the number \( K \),
their intersection is outside the channel \( C(u, v', K-\gf, K+\gf) \).
This implies \( x_{i+1}-x_{i}\ge\gf \) since \( \slope(u,v)\le 1 \).
On the other hand, by~\eqref{eq:diag-close}, the points \( w''_{i},w'_{i+1} \)
differ from \( w_{i},w_{i+1} \) by at most \( O(\slblb^{-1}\g) \).
It is easy to see from here that
\begin{align*}
    |\slope(w''_{i},w'_{i+1}) -\mu|&
    \le \aux_{0}\slblb^{-1}\g/\gf=\aux_{0}\slblb^{-1}\bub/\g,
\\|1/\slope(w''_{i},w'_{i+1}) -1/\mu|&\le \aux_{0}\slblb^{-3}\bub/\g
 \end{align*}
for some absolute constant \( \aux_{0} \) that can be computed.
Choosing \( \slopeincr>\aux_{0}+1 \), the definition of \( \slb_{i}^{*} \)
and the assumption on \( \mu \) imply that the slope condition
 \( \slb_{x}\le\slope(w''_{i},w'_{i})\le 1/\slb_{y} \) is satisfied.

The proof of \( u\leadsto w'_{1} \) and \( w''_{3}\leadsto v \) is similar,
taking into account \( x_{1}-u_{0} \ge 0.1\f  \) 
and \( v_{0}-x_{3} \ge 0.1\f  \).

  \end{pproof} 
\end{Proof}

\begin{Proof}[Proof of Lemma~\protect\ref{lem:approx}(Approximation)]
Recall that the lemma says that if a rectangle \( Q=\Rect^{\eps}(u,v) \) contains no
walls or traps of \( \cM^{*} \), is inner clean in \( \cM^{*} \) and
satisfies the slope condition
\( \slb^{*}_{x}\le\slope(u,v')\le 1/\slb^{*}_{y} \) with \( v' \) related to \( v \) as
above, then \( u\leadsto v \).

The proof started by recalling, in Lemma~\ref{lem:the-grate},
that walls of \( \cM \) in \( Q \) can be grouped to a
horizontal and a vertical sequence, whose members are well separated from each
other and from the sides of \( Q \).
Then it showed, in Lemma~\ref{lem:trap-cover},
that all traps of \( \cM \) are covered by certain
horizontal and vertical stripes called trap covers.
Walls of \( \cM \) and trap covers were called obstacles.

Next it showed, in Lemma~\ref{lem:approx.no-walls}, that in case there are no
traps or walls of \( \cM \) in \( Q \) then there is a path through \( Q \) that stays close
to the diagonal.

Next, a series of obstacles (walls or trap covers) was defined,
along with the points \( s_{1},s_{2},\dots \) that are obtained by the intersection points of
the obstacle with the diagonal, and projected to the \( x \) axis.
It was shown in Lemma~\ref{lem:one-of-three} that these obstacles are well separated
into groups of up to three.
Lemma~\ref{lem:three} showed how to pass each triple of obstacles.
It remains to conclude the proof.

For each pair of numbers \( s_{i}, s_{i+1} \) with 
\( s_{i+1} - s_{i} \ge 0.22\f \), define its midpoint \( (s_{i} + s_{i+1})/2 \).
Let \( t_{1} < t_{2} < \dots < t_{n} \) be the sequence of all these
midpoints.
With \( \mu=\slope(u,v') \), let us define the square
 \[
  S_{i} = \pair{t_{i}}{u_{1} + \mu(t_{i} - u_{0})} + 
  \clint{0}{\bub} \times \clint{-\bub}{0}.
 \]
By Remark~\ref{rem:distr}.\ref{i:middle-third}, each of these squares
contains a clean point \( p_{i} \).

 \begin{step+}{approx.middle}
\begin{sloppypar}
For \( 1 \le i < n \), the 
rectangle \( \Rect(p_{i}, p_{i+1}) \) satisfies the conditions
of Lemma~\ref{lem:three}, and therefore \( p_{i}\leadsto p_{i+1} \).
The same holds also for \( \Rect^{\eps}(u,p_{1}) \)
if the first obstacle is a wall, and for \( \Rect(p_{n},v) \)
if the last obstacle is a wall.
Here \( \eps=\uparrow,\to \) or nothing, depending on the nature of the original
rectangle \( \Rect^{\eps}(u,v) \).
\end{sloppypar}
 \end{step+}
 \begin{pproof} 
By Lemma~\ref{lem:one-of-three}, there are at most
three points of \( \set{s_{1}, s_{2}, \dots} \) between
\( t_{i} \) and \( t_{i+1} \).
Let these be \( s_{j_{i}}, s_{j_{i}+1}, s_{j_{i}+2} \).
Let \( t'_{i} \) be the \( x \) coordinate of \( p_{i} \), then 
\( 0 \le t'_{i} - t_{i} \le \bub \).
The distance of each \( t'_{i} \) from the closest point \( s_{j} \)
is at most \( 0.11\f - \bub \ge 0.1\f \).
It is also easy to check that 
\( p_{i}, p_{i+1} \) satisfy~\eqref{eq:triangle-cnd.three}, so 
Lemma~\ref{lem:three} is indeed applicable.
 \end{pproof} 

 \begin{step+}{approx.u}
We have \( u \leadsto p_{1} \) and \( p_{n}\leadsto v \).
 \end{step+}
 \begin{pproof}
If \( s_{1} \ge 0.1\f \), then the statement is proved by an
application of Lemma~\ref{lem:three}, so suppose \( s_{1} < 0.1\f \).
Then \( s_{1} \) belongs to a trap cover.

If \( s_{2} \) belongs to a wall then \( s_{2}\ge\f/3 \), so
\( s_{2}-s_{1}>0.23\f \).
If \( s_{2} \) also belongs to a trap cover then
the reasoning used in Lemma~\ref{lem:one-of-three} gives \( s_{2}-s_{1} > \f/4 \).
In both cases, a midpoint \( t_{1} \) was chosen between \( s_{1} \) 
and \( s_{2} \) with \( t_{1}-s_{1}>0.1\f \),
and there is only \( s_{1} \) between \( u \) and \( t_{1} \).

If the trap cover belonging to \( s_{1} \) is closer to \( u \) than 
\( \g-6\bub \) then the fact that \( u \) is clean in 
\( \cM^{*} \) implies that it contains a large trap-free region where it is
easy to get through.

If it is at a distance \( \ge \g-6\bub \) from \( u \)
then we will pass through it, going from \( u \) to \( p_{1} \) similarly to
Part~\ref{step:three.x-free.trapc.vert} of the proof of Lemma~\ref{lem:three},
but using case \( j=1 \) of Lemma~\ref{lem:if-correl}, in place of \( j=2 \).
This means using \( \L_{1}=29\slblb^{-1}\bub \) in place of \( \L_{2} \).
As a consequence, we will have
 \( |x-x_{1}|+ |y-y_{1}| = O(\slblb^{-1}\bub) \) in place of~\eqref{eq:diag-close}.
This makes a change of slope by \( O(\slblb^{-1}\bub/\g) \), so an appropriate
choice of the constant \( \slopeincr \) finishes the proof just as in 
part~\ref{step:three.bound} of the proof of Lemma~\ref{lem:three}.

The relation \( p_{n}\leadsto v \) is shown similarly.
 \end{pproof} 
 \end{Proof}

\section{Proof of the main lemma}
\label{sec:main-proof}

Lemma~\ref{lem:main} asserts the existence of a sequence of mazeries \( \cM^{k} \)
such that certain inequalities hold.
The construction of \( \cM^{k} \) is complete by the definition of \( \cM^{1} \) in
Example~\ref{xmp:base} and the scale-up algorithm of
Section~\ref{sec:plan}, after fixing all parameters in Section~\ref{sec:params}.

We will prove, by induction, that every structure \( \cM^{k} \) is a mazery.
Lemma~\ref{lem:base} shows this for \( k=1 \).
Assuming that it is true for all \( i \le k \), we prove it for \( k+1 \).
The dependency properties in 
Condition \ref{cond:distr}.\ref{i:distr.indep} are satisfied according to
Lemma~\ref{lem:distr.indep}.
The combinatorial properties in
Condition \ref{cond:distr}.\ref{i:distr.combinat} have been proved in
Lemmas~\ref{lem:cover} and~\ref{lem:clean}.
The reachability property in Condition \ref{cond:distr}.\ref{i:distr.reachable}
is satisfied via Lemma~\ref{lem:approx}.

\begin{sloppypar}
The trap probability upper bound in
Condition \ref{cond:distr}.\ref{i:distr.trap-ub}
has been proved in Lemma~\ref{lem:trap-scale-up}.
The wall probability upper bound in Condition \ref{cond:distr}.\ref{i:distr.wall-ub}
has been proved in Lemma~\ref{lem:all-wall-ub}.
The cleanness probability lower bounds in
Condition \ref{cond:distr}.\ref{i:distr.ncln}
have been proved in Lemma~\ref{lem:ncln-ub}.
The hole probability lower bound in
Condition \ref{cond:distr}.\ref{i:distr.hole-lb} has been proved in 
Lemmas~\ref{lem:emerg-hole-lb}, \ref{lem:all-compound-hole-lb}
and~\ref{lem:heavy-hole-lb}.
\end{sloppypar}

Inequality~\eqref{eq:main.big-sum} 
of Lemma~\ref{lem:main} is proved in Lemma~\ref{lem:0.6}.

\section{Conclusions}
 
The complex hierarchical technique has been used now to prove three results of
the dependent percolation type: those in~\cite{GacsChat04}, \cite{GacsWalks11}
and the present one.
Each of these proofs seems too complex for the result proved, and to give only a
very bad estimate of the bound on the critical value of
the respective parameter.
In this, they differ from the related results on the undirected percolation 
in~\cite{WinklerCperc99} and~\cite{BalBollobStacCperc99} (on the other hand,
all three directed
percolations exhibit power-law behavior).

For the other two problems, given that their original form relates to 
scheduling, it was natural to ask about possible extensions of the results to more
than two sequences.
I do not see what would be a natural extension of the embedding problem in this direction.

\subsection*{Acknowledgement}

I am grateful to the anonymous referees (especially Referee B) who,
besides catching innumerable bugs, suggested some nice sharpenings
and simplifications.
Thanks are still due also to the referees of the earlier papers~\cite{GacsChat04} 
and~\cite{GacsWalks11}.


\end{document}